\documentclass[11pt,reqno]{amsart}

\usepackage{graphicx}
\usepackage[top=20mm, bottom=20mm, left=30mm, right=30mm]{geometry}
\usepackage[colorlinks=true,citecolor=black,linkcolor=black,urlcolor=magenta]{hyperref}

\usepackage{fancyvrb}
\usepackage{ulem} 
\usepackage[T1]{fontenc} 
\usepackage{pbsi} 
\usepackage{url} 
\usepackage{calc} 
\usepackage{fancyvrb}

\usepackage{local-math-commands}

\usepackage{enumitem}
\setlist[itemize]{leftmargin=0.35in}

\usepackage{diagbox}

\usepackage{ifthen}
\newcommand{\Hn}[2]{
     \ifthenelse{\equal{#2}{1}}{H_{#1}}{H_{#1}^{\left(#2\right)}}
}

\newcommand{\QBinomial}[3]{\gkpSI{#1}{#2}_{#3}} 
\newcommand{\QPochhammer}[3]{\left(#1; #2\right)_{#3}}

\DeclareMathOperator{\ab}{ab}
\DeclareMathOperator{\ConvP}{P}
\DeclareMathOperator{\ConvQ}{Q}
\DeclareMathOperator{\IC}{IC}
\DeclareMathOperator{\RHSOp}{RHS}
\DeclareMathOperator{\LHSOp}{LHS}

\title{
       Continued Fractions for Square Series Generating Functions} 

\author{Maxie D. Schmidt \\ 
        \href{mailto:maxieds@gmail.com}{maxieds@gmail.com}
        } 
\address{University of Washington \\
        Department of Mathematics \\ 
        Padelford Hall \\ 
        Seattle, WA 98195 \\ 
        USA 
        } 
\date{Revised 2017.02.17-v2}

\allowdisplaybreaks

\begin{document}

\begin{abstract}
We consider new series expansions for 
variants of the so-termed ordinary geometric square series 
generating functions originally defined in the recent article titled 
``\emph{Square Series Generating Function Transformations}'' 
(arXiv: 1609.02803). 
Whereas the original square series transformations article 
adapts known generating function 
transformations to construct integral representations for these 
square series functions enumerating the square powers of $q^{n^2}$ for some 
fixed non-zero $q$ with $|q| < 1$, we study the 
expansions of these special series 
through power series generated by Jacobi-type continued fractions, or J-fractions. 
We prove new exact expansions of the $h^{th}$ convergents to these 
continued fraction series and show that the limiting case of these 
convergent generating functions exists as $h \rightarrow \infty$. 

We also prove new infinite $q$-series representations of special square series 
expansions involving square-power 
terms of the series parameter $q$, the $q$-Pochhammer symbol, and 
double sums over the $q$-binomial coefficients. 
Applications of the new results we prove within the article include new 
$q$-series representations for the ordinary generating functions of the 
special sequences, $r_p(n)$, and $\sigma_1(n)$, as well as parallels to the 
examples of the new integral representations for theta functions, 
series expansions of infinite products and partition function 
generating functions, and related unilateral special function series 
cited in the first square series transformations article. 


\bigskip\noindent
\underline{MSC Subject Codes}: \emph{05A15; 11Y65; 11B65; 40A15.} 

\end{abstract} 

\maketitle

\section{Introduction} 

\subsection{Continued Fraction Expansions of Ordinary Generating Functions} 

\emph{Jacobi-type continued fractions}, or \emph{J-fractions}, 
correspond to power series defined by 
infinite continued fraction expansions of the form 
\begin{align} 
\label{eqn_J-Fraction_Expansions} 
J_{\infty}\left(\{c_i\}, \{\ab_i\}; z\right) 
     & = 
     \cfrac{1}{1-c_1z-\cfrac{\ab_2 z^2}{1-c_2z- 
     \cfrac{\ab_3 z^2}{\cdots}}} \\ 
\notag 
     & = 
     1 + c_1 z + \left(\ab_2+c_1^2\right) z^2 + 
     \left(2\ab_2 c_1+c_1^3+\ab_2 c_2\right) z^3 + 
     \cdots, 
\end{align} 
for some sequences $\{ c_i \}_{i=1}^{\infty}$ and 
$\{ \ab_i \}_{i=2}^{\infty}$, and 
some (typically formal) series variable $z \in \mathbb{C}$ 
\citep[\cf \S 3.10]{NISTHB} \citep{WALL-CFRACS}. 
The formal series enumerated by special cases of the truncated and infinite 
J-fraction series of this form include \emph{ordinary} 
(as opposed to typically closed-form \emph{exponential}) 
\emph{generating functions} 
for many one and two-index combinatorial sequences including the 
generalized factorial functions studied in the references 
and in the results from Flajolet's articles. 
The references provide infinite formal continued fraction 
series for the typically divergent ordinary generating functions of many 
special sequences, factorial functions, and combinatorial triangles 
arising in applications 
\citep{FLAJOLET80B,FLAJOLET82,GFLECT,MULTIFACT-CFRACS}. 

\subsection{J-Fraction Expansions of Ordinary Square Series Generating Functions} 

In this article, we consider a $q$-series-related variant of the 
J-fractions in \eqref{eqn_J-Fraction_Expansions} 
enumerating the square-power terms of $q^{n^2}$, 
which we have discovered computationally 
(see \citep{SQSERIESMDS} for integral representations). 
In particular, for non-zero $q, z \in \mathbb{C}$ such that $|q|, |z| < 1$, 
we define the following Jacobi-type J-fraction, $J_{\infty}^{[\sq]}(q, z)$, 
for the ordinary square series generating function defined in 
\citep{SQSERIESMDS} as follows: 
\begin{align} 
\label{eqn_SquareSeries_J-Fraction_Expansions} 
J_{\infty}^{[\sq]}(q, z) & = 
     \cfrac{1}{1-qz-\cfrac{q^2(q^2-1) z^2}{1-q(q^4+q^2-1)z- 
     \cfrac{q^8 (q^4-1) z^2}{1-q^3(q^6+q^4-1)z-\cfrac{q^{14}(q^6-1) z^2}{\cdots.}}}} 
\end{align} 

\begin{definition}[Component Sequences and Convergent Functions]  
\label{def_ciq_abiq_Phqz_Qhqz}
We define the component sequence functions, 
$c_h \equiv c_h^{[\sq]}(q)$ and $\ab_h \equiv \ab_h^{[\sq]}(q)$, as follows 
for $h \geq 0$: 
\begin{align} 
\label{eqn_chq_subseq_def} 
c_h^{[\sq]}(q) & := 
     \begin{cases} 
          q^{2h-3} \cdot \left(q^{2h} + q^{2h-2} - 1\right) & 
          \text{ if $h \geq 2$; } \\ 
          q & \text{ if $h = 1$; } \\ 
          1 & \text{ if $h = 0$, } 
     \end{cases} \\ 
\label{eqn_abhq_subseq_def} 
\ab_h^{[\sq]}(q) & := 
     \begin{cases} 
          q^{6h-10} \cdot \left(q^{2h-2} - 1\right) & 
          \text{ if $h \geq 2$; } \\ 
          0 & \text{ otherwise. }
     \end{cases} 
\end{align} 
We then define the $h^{th}$ convergent functions, 
$\Conv_h(q, z) := \ConvP_h(q, z) / \ConvQ_h(q, z)$, recursively through the 
component numerator and denominator functions given by\footnote{ 
     \underline{Notation}: 
     \emph{Iverson's convention} compactly specifies 
     boolean-valued conditions and is equivalent to the 
     \emph{Kronecker delta function}, $\delta_{i,j}$, as 
     $\Iverson{n = k} \equiv \delta_{n,k}$. 
     Similarly, $\Iverson{\mathtt{cond = True}} \equiv 
     \delta_{\mathtt{cond}, \mathtt{True}}$ in the remainder of the article. 
} 
\begin{align} 
\label{eqn_ConvFn_PhzQhz_rdefs} 
\ConvP_h(q, z) & = (1-c_{h}^{[\sq]}(q) \cdot z) \ConvP_{h-1}(q, z) - 
     \ab_{h}^{[\sq]}(q) \cdot z^2 \ConvP_{h-2}(q, z) + \Iverson{h = 1} \\ 
\notag 
\ConvQ_h(q, z) & = (1-c_{h}^{[\sq]}(q) \cdot z) \ConvQ_{h-1}(q, z) - 
     \ab_{h}^{[\sq]}(q) \cdot z^2 \ConvQ_{h-2}(q, z) + 
     (1-c_{1}^{[\sq]}(q) \cdot z) \Iverson{h = 1} + \Iverson{h = 0}. 
\end{align} 
\end{definition} 
We claim, and subsequently prove in Section \ref{Section_Proofs}, 
that the truncated $(h+1)$-order power series expansions of the $h^{th}$ 
convergent functions, $\Conv_h(q, z)$, defined in 
Definition \ref{def_ciq_abiq_Phqz_Qhqz} generate the so-termed 
``\emph{square series}'' powers defined in \citep{SQSERIESMDS} as 
\begin{align*} 
\Conv_h(q, z) & = 
     \cfrac{1}{1-qz-\cfrac{q^2(q^2-1) z^2}{1-q(q^4+q^2-1)z- 
     \cfrac{q^8 (q^4-1) z^2}{1-q^3(q^6+q^4-1)z-\cfrac{q^{14}(q^6-1) z^2}{ 
     \cfrac{\cdots}{1-q^{2h-3}(q^{2h}+q^{2h-2}-1) z}}}}} \\ 
     & = 
     1 + q z + q^4 z^2 + q^9 z^3 + \cdots + q^{h^2} z^h + 
     \sum_{n=h+1}^{\infty} \widetilde{e}_{h,n}(q) \cdot z^n, 
\end{align*} 
for all $n, h \in \mathbb{Z}^{+}$. 
More precisely, we state the next result as our main theorem which 
we carefully prove in 
Section \ref{subSection_ProofOfMainTheorem} of the article below. 

\begin{theorem}[J-Fraction Expansions of Ordinary Square Series Generating Functions] 
\label{theorem_main_theorem} 
For all $h \geq 1$ and every $0 \leq n \leq h$, we have that 
$[z^n] \Conv_h(q, z) = q^{n^2}$. 
\end{theorem} 

For $0 < |q|, |z| < 1$, we show that the 
limiting convergent function, denoted by 
\begin{align*} 
\tag{Ordinary Geometric Square Series} 
J_{\infty}^{[\sq]}(q, z) 
     := \sum_{n \geq 0} q^{n^2} z^n = 
     \lim_{h \longrightarrow \infty}\ \Conv_h(q, z), 
\end{align*} 
which corresponds to the integral representations of the 
\emph{ordinary geometric square series} cases defined in 
\citep{SQSERIESMDS} given by 
\begin{align*} 
\tag{Ordinary Geometric Square Series Integral Representation} 
J_{\infty}^{[\sq]}(q, z) & = \int_0^{\infty} 
     \frac{2 e^{-t^2/2}}{\sqrt{2\pi}} \left[ 
     \frac{1 - z \cosh\left(t \sqrt{2 \Log(q)}\right)}{ 
     z^2 - 2z \cosh\left(t \sqrt{2 \Log(q)}\right) + 1} 
     \right] dt, 
\end{align*} 
exists, and that the limiting cases of this convergent function series 
converge uniformly as a function of $z$ with respect to the sequence of $h^{th}$ 
partial sums of \eqref{eqn_Intro_GeomSqSeries_NewSeriesExp} defined below. 

Thus we are able to prove new infinite $q$-series identities generating 
special cases of $J_{\infty}^{[\sq]}(q, c \cdot q^{p})$ in 
Section \ref{Section_Proofs}, 
which arise in many applications to theta functions and 
special function series. 
The primary new form of the infinite q-series expansions we prove in 
Section \ref{Section_Proofs} of the article has the form 
\begin{align} 
\label{eqn_Intro_GeomSqSeries_NewSeriesExp}
J_{\infty}^{[\sq]}(q, z) & = \sum_{i=1}^{\infty} 
     \frac{(-1)^{i-1} q^{(3i-4)(i-1)} \QPochhammer{q^2}{q^2}{i-1} z^{2i-2}}{ 
     \sum\limits_{0 \leq j \leq n < 2i} 
     \QBinomial{i}{j}{q^2} \QBinomial{i-1}{n-j}{q^2} q^{2j} \cdot 
     \left(-q^{2i-3} \cdot z\right)^n}. 
\end{align} 
Moreover, the $h^{th}$ partial sums of the infinite $q$-series expansion in 
\eqref{eqn_Intro_GeomSqSeries_NewSeriesExp} corresponding to the $h^{th}$ 
convergent function, $\Conv_h(q, z)$, is always at least $(h+1)$-order accurate in 
generating the coefficients, $q^{n^2}$, of $z^n$ for $0 \leq n \leq h$. 
In addition, we show that we can 
differentiate these new infinite square series expansions termwise with 
respect to $z$ thus allowing us to generate series for 
higher-order derivatives of the \emph{Jacobi theta functions}, 
which are generated as particular cases of the geometric square series, 
$\vartheta_{0,m}(a, b; q, z) := \sum_{n} (an+b)^m q^{n^2} z^n$, 
in \citep{SQSERIESMDS} and which similarly arise in many 
well-known identities and classical expansions of special products and 
partition function generating functions. 

\begin{table}[ht]
\centering 
\begin{tabular}{|c|l|} \hline 
$h$ & $\ConvP_h(q, z)$ \\ \hline 
1 & $1$ \\
2 & $1-q \left(q^4+q^2-1\right) z$ \\
3 & $1-q \left(1+q^2\right) \left(-1+q^2+q^6\right) z+q^4 \left(1-q^2-q^4+q^8+q^{10}\right) z^2$ \\ 
4 & $1-q \left(1-q+q^2\right) \left(1+q+q^2\right) \left(-1+q^2+q^8\right) z+q^4 \left(1-q^4-q^6-q^8+q^{12}+2 q^{14}+q^{16}+q^{18}\right) z^2$ \\ 
  & $\phantom{1}-q^9 \left(-1+q^2+q^4-2 q^{10}+q^{16}+q^{18}\right) z^3$ \\ 
5 & $1-q \left(1+q^2\right) \left(1+q^4\right) \left(1-q^2+q^4\right) \left(-1+q^4+q^6\right) z$ \\ 
  & $\phantom{1}+q^4 \left(1+q^2\right) \left(1-q+q^2\right) \left(1+q+q^2\right) \left(1-q^2+q^4\right) \left(1-q^2-q^6+q^{12}+q^{16}\right) z^2$ \\ 
  & $\phantom{1}-q^9 \left(1+q^2\right) \left(-1+q^2+q^6-2 q^{14}-q^{18}+q^{20}+2 q^{24}+q^{28}\right) z^3$ \\ 
  & $\phantom{1}+q^{16} \left(1-q^2-q^4+q^{10}+q^{12}+q^{14}-q^{16}-q^{18}-q^{20}+q^{26}+q^{28}\right) z^4$ \\ 
\hline 
\end{tabular} 

\bigskip 
\caption{The Numerator Convergent Functions, $\ConvP_h(q, z)$} 
\label{table_ConvFnPhz_spcases} 

\end{table} 

\begin{table}[ht]
\centering 
\begin{tabular}{|c|l|} \hline 
$h$ & $\ConvQ_h(q, z)$ \\ \hline 
0 & $1$ \\ 
1 & $1 - q z$ \\
2 & $1 - q^3 (1 + q^2) z + q^6 z^2$ \\
3 & $1 - q^5 (1 - q + q^2) (1 + q + q^2) z + 
     q^{10} (1 - q + q^2) (1 + q + q^2) z^2 - q^{15} z^3$ \\ 
4 & $1 - q^7 (1 + q^2) (1 + q^4) z + 
     q^{14} (1 - q + q^2) (1 + q + q^2) (1 + q^4) z^2$ \\ 
  & $\phantom{1}- 
     q^{21} (1 + q^2) (1 + q^4) z^3 + q^{28} z^4$ \\ 
5 & $1 - q^9 (1 - q + q^2 - q^3 + q^4) (1 + q + q^2 + q^3 + q^4) z$ \\ 
  & $\phantom{1}+ 
     q^{18} (1 + q^4) (1 - q + q^2 - q^3 + q^4) (1 + q + q^2 + q^3 + q^4) z^2$ \\ 
  & $\phantom{1}- 
     q^{27} (1 + q^4) (1 - q + q^2 - q^3 + q^4) (1 + q + q^2 + q^3 + 
     q^4) z^3$ \\ 
  & $\phantom{1}+ 
     q^{36} (1 - q + q^2 - q^3 + q^4) (1 + q + q^2 + q^3 + q^4) z^4 - 
     q^{45} z^5$ \\ 
\hline 
\end{tabular} 

\bigskip 
\caption{The Denominator Convergent Functions, $\ConvQ_h(q, z)$} 
\label{table_ConvFnQhz_spcases} 

\end{table} 

Tables of the expansions of the component convergent functions, 
$\ConvP_h(q, z)$ and $\ConvQ_h(q, z)$, for the first few special cases of 
$h \geq 1$ are given in 
Table \ref{table_ConvFnPhz_spcases} and 
Table \ref{table_ConvFnQhz_spcases}. 
Proposition \ref{prop_DenomConvFn_Seq_Formula} and 
Proposition \ref{prop_NumConvFn_Seq_Formula} 
stated in Section \ref{Section_Proofs} of the article below 
provide more general exact formulas satisfied by these 
convergent function sequences for all $h \geq 1$. 
Since the \emph{sums of squares functions}, $r_p(n)$, are generated by the 
$p^{th}$ powers of the \emph{Jacobi theta function}, 
$\vartheta_3(q) = 1 + 2 q + 2 q^4 + 2 q^9 + \cdots$, 
we are particularly interested in the properties of the convergent function 
special cases, $\widetilde{\vartheta}_{3,h}(q) := 1 + 2q \cdot \Conv_h(q, q^2)$, 
in the context of the applications briefly motivated in the next subsection 
(see Table \ref{table_ConvFnConvhq2z_spcases} on page 
\pageref{table_ConvFnConvhq2z_spcases}). 

\subsection{Generating the Sums of Squares Functions} 

In Section \ref{Section_Applications}, 
we use the new square series J-fraction representation for 
$\widetilde{\vartheta}_{3,h}(q)$, which we will prove in the next section, 
to expand new forms of 
congruences and generating functions for the 
\emph{sums of squares functions}, $r_2(n)$, and more generally for 
$r_p(n)$ when $p \geq 3$. 
For a fixed integer $p \geq 2$, the function, $r_p(n)$, is defined to be the 
number of solutions to the equation 
\begin{equation*} 
n = x_1^2 + x_2^2 + \cdots + x_p^2, 
\end{equation*} 
where the integers $x_i$ are taken to be positive, negative, or zero-valued. 
For each $p \geq 2$, these sequences are generated over $q$ by the following 
$p^{th}$ powers of the \emph{Jacobi theta function}, $\vartheta_3(q)$ 
\citep{NISTHB,HARDYANDWRIGHT}: 
\begin{align*} 
\sum_{n \geq 0} r_p(n) q^n & = \vartheta_3(q)^{p} = \left(
     1 + 2q \times \sum_{n=0}^{\infty} q^{n(n+2)}\right)^{p}. 
\end{align*} 
If we let $\Conv_h(q, z) := P_h(q, z) / Q_h(q, z)$ denote the $h^{th}$ 
convergent function to our new J-fraction expansion, as in 
Definition \ref{def_ciq_abiq_Phqz_Qhqz}, and let 
$E_h(z^m) := z^m \cdot \Iverson{0 \leq m \leq h}$ denote the 
\emph{erasing operator}, or $(h+1)$-order power series truncation operator in $z$, 
we can generate the sums of squares functions by 
\begin{align*} 
\sum_{0 \leq n \leq h} q^{n^2} \left(q^{2} z\right)^n & = 
     E_h\left[\left(\frac{Q_h(q, q^2 z) + 2q P_h(q, q^2 z)}{ 
     Q_h(q, q^2 z)}\right)^{p} \right]. 
\end{align*} 
In Section \ref{subSection_rpn_NewOGF_reprs}, 
we give a new explicit formula for the generating functions 
of $r_p(n)$ for all $p \geq 2$ in the form of 
\begin{align*}
r_p(n) & = [q^n] \left(1 + 2q \times \sum_{i=1}^{\infty} 
     \frac{(-1)^{i-1} q^{3i(i-1)} \QPochhammer{q^2}{q^2}{i-1}}{ 
     \sum\limits_{0 \leq n < 2i} \left(\sum\limits_{0 \leq j \leq n} 
     \QBinomial{i}{j}{q^2} 
     \QBinomial{i-1}{n-j}{q^2} q^{2j} \right) (-q^{2i-1})^n} 
     \right)^{p}. 
\end{align*}
Other corollaries of the results we prove in 
Section \ref{Section_Proofs} are given in 
Section \ref{Section_Congruences_GFs}, including applications to new 
congruences for the sums of squares functions, as well as a 
new approach to generating the 
odd-indexed \emph{sum of divisors function}, 
$r_4(2k+1) \equiv \sigma_1(2k+1)$, 
through our new modified $q$-series generating function constructions. 

We will also show in Section \ref{Section_Applications} that we can 
expand the generating function for the complete sequence of $\sigma_1(n)$, 
which is well-known to satisfy 
$\sigma_1(n) = -[q^{n-1}] d / dq\left[\Log \QPochhammer{q}{q}{\infty}\right]$, 
using our new results applied to the series for the infinite 
\emph{$q$-Pochhammer symbol} given by 
\begin{align*} 
\QPochhammer{q}{q}{\infty} & = 
     1 - q \times \sum_{i=1}^{\infty} 
     \frac{(-1)^{i-1} q^{(9i-2)(i-1)/2} \QPochhammer{q^3}{q^3}{i-1}}{ 
     \sum\limits_{0 \leq j \leq n < 2i} \QBinomial{i}{j}{q^3} 
     \QBinomial{i-1}{n-j}{q^3} q^{3j} \cdot q^{(3i-2) n}} \\ 
     & \phantom{=1\ } - 
     q^2 \times \sum_{i=1}^{\infty} 
     \frac{(-1)^{i-1} q^{(9i+2)(i-1)/2} \QPochhammer{q^3}{q^3}{i-1}}{ 
     \sum\limits_{0 \leq j \leq n < 2i} \QBinomial{i}{j}{q^3} 
     \QBinomial{i-1}{n-j}{q^3} q^{3j} \cdot q^{(3i-1) n}}. 
\end{align*} 
Another $q$-series expansion generating the \emph{partition function}, $p(n)$, 
is similarly expanded through our new results as the reciprocal of the 
series in the previous equation. 

\subsection{Other Applications of the New Results Proved in the Article} 

The results of the applications stated in 
Section \ref{Section_Applications} 
provide new series representations for 
generating functions of special sequences, including the 
\emph{sums of squares functions}, $r_p(n)$, and the 
\emph{sum of divisors} (\emph{divisor sigma}) function, $\sigma_1(n)$, 
as well as new infinite series for the (higher-order) 
Jacobi theta functions, and other special case series results 
such as for the unilateral series expansion of the 
generalized \emph{two-variable Ramanujan theta function}, $f(a, b)$, 
expanded by both its integral representation proved in \citep{SQSERIESMDS} 
and our new series result in \eqref{eqn_Intro_GeomSqSeries_NewSeriesExp} as 
\begin{align*} 
f(a, b) & = 
     1 + \int_0^{\infty} \frac{2a e^{-t^2/2}}{\sqrt{2\pi}}\left[ 
     \frac{1 - a \sqrt{ab} \cosh\left(\sqrt{\Log(ab)} t\right)}{ 
     a^3 b - 2a \sqrt{ab} \cosh\left(\sqrt{\Log(ab)} t\right) + 1} 
     \right] dt \\ 
\notag 
   & \phantom{=\ 1 } + 
     \int_0^{\infty} \frac{2b e^{-t^2/2}}{\sqrt{2\pi}}\left[ 
     \frac{1 - b \sqrt{ab} \cosh\left(\sqrt{\Log(ab)} t\right)}{ 
     a b^3 - 2b \sqrt{ab} \cosh\left(\sqrt{\Log(ab)} t\right) + 1} 
     \right] dt \\ 
     & = 
     1 + \sum_{c \in \{a, b\}} \sum_{i=1}^{\infty} 
     \frac{c \cdot (-1)^{i-1} (ab)^{(3i-2)(i-1)} 
     \QPochhammer{ab}{ab}{i-1} c^{2i-2}}{\sum\limits_{0 \leq j \leq n < 2i} 
     \QBinomial{i}{j}{ab} \QBinomial{i-1}{n-j}{ab} (ab)^j \cdot 
     \left(-(ab)^{i-1} \cdot c\right)^{n}}. 
\end{align*} 
We also prove new convergent-function-based infinite $q$-series 
expansions of the Jacobi theta functions, 
$\vartheta_i(q, z)$, for $i := 1,2,3,4$ 
in Section \ref{subSection_JThetaFns_Newq-Series_reprs}. 
For example, the next special constant values of 
\emph{Ramanujan's $\varphi$-function and $\psi$-function} cited as 
applications in \citep{SQSERIESMDS}, are expanded by the following 
new series and corresponding known integral representations from the 
reference: 
\begin{align} 
\label{eqn_PairOfExamples_ExplicitSpConstantValues_JThetaFnSpCases} 
\vartheta_3\left(0, e^{-5\pi}\right) & \equiv 
\frac{\pi^{1/4}}{\Gamma\left(\frac{3}{4}\right)} \cdot 
     \frac{\sqrt{5 + 2 \sqrt{5}}}{5^{3/4}} \\ 
\notag 
     & = 
     1 + \int_0^{\infty} \frac{e^{-t^2/2}}{\sqrt{2\pi}} \left[ 
     \frac{4 e^{5\pi} \left(e^{10\pi} - \cos\left(\sqrt{10 \pi} t\right) 
     \right)}{e^{20\pi} - 2 e^{10\pi} \cos\left(\sqrt{10 \pi} t\right) + 1} 
     \right] dt \\ 
\notag 
     & = 
     1 + 2 e^{-5 \pi} \times \sum_{i=1}^{\infty} 
     \frac{(-1)^{i-1} e^{-15\pi i(i-1)} 
     \QPochhammer{e^{-10\pi}}{e^{-10\pi}}{i-1}}{ 
     \sum\limits_{0 \leq j \leq n < 2i} 
     \QBinomial{i}{j}{e^{-10\pi}} \QBinomial{i-1}{n-j}{e^{-10\pi}} \times 
     e^{-10\pi \cdot j} \cdot \left(-e^{-5\pi (2i-1)}\right)^{n}} \\ 
\notag 
\frac{e^{\pi / 8}}{2} \cdot \vartheta_2\left(0, e^{-\pi / 2}\right) & \equiv 
     \frac{\pi^{1/4}}{\Gamma\left(\frac{3}{4}\right)} \cdot 
     \frac{\left(\sqrt{2} + 1\right)^{1/4} e^{\pi / 16}}{2^{7/16}} \\ 
\notag 
     & = 
     \int_0^{\infty} \frac{e^{-t^2/2}}{\sqrt{2\pi}} \left[ 
     \frac{\cos\left(\sqrt{\frac{\pi}{2}} t\right) - e^{\pi / 4}}{ 
     \cos\left(\sqrt{\frac{\pi}{2}} t\right) - \cosh\left(\frac{\pi}{4}\right)} 
     \right] dt \\ 
\notag 
     & = 
     \sum_{i=1}^{\infty} \frac{(-1)^{i-1} e^{-\pi (3i-2)(i-1) / 2} 
     \QPochhammer{e^{-\pi}}{e^{-\pi}}{i-1}}{\sum\limits_{0 \leq j \leq n < 2i} 
     \QBinomial{i}{j}{e^{-\pi}} \QBinomial{i-1}{n-j}{e^{-\pi}} \times 
     e^{-\pi \cdot j} \cdot \left(-e^{-\pi(i-1)}\right)^{n}}. 
\end{align} 
We cite other special case square series variants treated as 
notable examples in \citep{SQSERIESMDS} as applications of 
Proposition \ref{prop_InfiniteGeomSqSeries_qSeries_repr} and 
Corollary \ref{cor_InfiniteGeomSqSeries_qSeries_repr_anpb_derivs} 
in Section \ref{Section_Applications} of the article below. 

With the exception of the first examples of the new forms of the 
$q$-series generating functions enumerating the sequences, $r_p(n)$ and 
$\sigma_1(n)$, established in Section \ref{subSection_rpn_NewOGF_reprs}, 
we mostly follow the examples cited as noteworthy applications of the 
new results for square series integral representations 
proved in \citep{SQSERIESMDS} in Section \ref{Section_Applications}. 
Further special cases of a broader set of results generalizing 
\eqref{eqn_PairOfExamples_ExplicitSpConstantValues_JThetaFnSpCases} 
in terms of the Jacobi theta functions lead to other 
corollaries providing new series expansions of special products, 
new series representations of explicit constants involving real-valued 
multiples of reciprocals of the rational gamma function cases, 
$\Gamma\left(\frac{1}{4}\right)$, $\Gamma\left(\frac{1}{3}\right)$, and 
$\Gamma\left(\frac{3}{4}\right)$, and 
integral-based representations of Mellin transforms defining 
known multiples of the \emph{Riemann zeta function}, $\zeta(s)$, when 
$\Re(s) > 2$. 

\subsection{Significance of the New Results in the Article} 

The new $q$-series expansions of special sequence generating functions and 
series for theta functions given by \eqref{eqn_Intro_GeomSqSeries_NewSeriesExp} 
are decidedly new in form and unlike other known series expansions for the 
special case examples we cite within the article. 
What is interesting to note about the proofs we obtain for these new 
series is that we apply the seemingly unrelated theory of 
J-fractions, and more general results on continued fractions, to a 
computationally-discovered sequence implicit to the typically 
combinatorially motivated definition of \eqref{eqn_J-Fraction_Expansions} to 
reach our new results for many special functions. 
As such, the rational $h^{th}$ partial sums with respect to $q$ and $z$ 
of the infinite series on the 
right-hand-side of \eqref{eqn_Intro_GeomSqSeries_NewSeriesExp} 
generate the expected series coefficients over 
$z$ (or $q$ if $z \equiv z(q)$ depends only on $q$) up to order $h+1$. 
The existence of the infinite series expansion follows by applying 
\emph{Pringsheim's theorem} to our computationally-guided ``guesses`` at the 
correctness of the sequences we initially defined in 
\eqref{eqn_chq_subseq_def} (see Remark \ref{remark_StrongerStmts_OfThe_Theorem} on 
page \pageref{remark_StrongerStmts_OfThe_Theorem}). 

\subsubsection*{Extensions of Our Computational Method} 

We note that our method of discovering the form of the infinite square series 
J-fraction expansion, $J_{\infty}^{[\sq]}(q, z)$, is easily replicated to find 
further analogous, if not sometimes related, new expansions for 
other special $q$-series generating functions. 
In the concluding remarks given in 
Section \ref{subSection_Conclusions_Remarks}, we point out several additional 
computationally discovered sequence variants implicit to the J-fractions 
defined in \eqref{eqn_J-Fraction_Expansions} and provide a table of these 
identified component sequences generating a few notable $q$-series variants, 
including \emph{$q$-exponential functions}, special cases of 
\emph{$q$-hypergeometric functions}, and alternate known series related to the 
generating function for the partition function $p(n)$ 
(see Table \ref{table_Conclusions_SpecialJ-Fraction_ConjSubSeq_Results} on page 
\pageref{table_Conclusions_SpecialJ-Fraction_ConjSubSeq_Results}). 
We also provide as an example, but do not rigorously prove, 
another infinite $q$-series-like infinite sum analog to 
\eqref{eqn_Intro_GeomSqSeries_NewSeriesExp} which provides a generating function 
for the $q$-Pochhammer symbol, 
$\QPochhammer{a}{q}{n} = (1-a)(1-aq)\cdots(1-a q^{n-1})$. 
Another example which motivates further study of our methods described within this 
article, but that we do not explicitly cite closed-form sums for, 
provides generating functions for the \emph{divisor function}, 
$d(n) = \sum_{d | n} 1$, in the form of\footnote{ 
     Additional identities of K. Dilcher from his article titled 
     ``\emph{Some $q$-series identities related to divisor functions}'' 
     provide expansions of the series coefficients of 
     $\QPochhammer{q}{q}{\infty} \times 
      \sum_n n^{\alpha+1} q^n / \QPochhammer{q}{q}{n}$ 
     when $\alpha \in \mathbb{Z}^{+}$. We can then generate these series by 
     considering higher-order derivatives of our new J-fraction series for 
     $1 / \QPochhammer{z}{q}{\infty}$ with respect to $z$ in the known 
     Stirling number transformations proved in \citep[\S 2]{SQSERIESMDS}. 
} 
\begin{align*} 
\sum_{n=1}^{\infty} d(n) q^n & = \QPochhammer{q}{q}{\infty} \times 
     \sum_{n=1}^{\infty} \frac{n^{} q^n}{\QPochhammer{q}{q}{n}}, 
\end{align*} 
where we can generate the terms of 
$n^{} / \QPochhammer{q}{q}{n}$ through first-order derivatives of the 
newly identified series for $1 / \QPochhammer{z}{q}{\infty}$ 
with respect to $z$ defined in Section \ref{subSection_Conclusions_Remarks}. 

\section{Proofs of the J-Fraction Representation} 
\label{Section_Proofs} 

\subsection{Exact Formulas for the Convergent Functions} 

\begin{definition}[$q$-Pochhammer Symbols and the $q$-Binomial Coefficients] 
\label{def_qBinom_qPH}
For fixed non-zero $z, q \in \mathbb{C}$ and integers $n \geq 0$, 
we define the \emph{$q$-Pochhammer symbol}, $\QPochhammer{z}{q}{n}$, 
as the product \citep[\S 17.2]{NISTHB} 
\begin{align*} 
\tag{$q$-Pochhammer Symbol} 
\QPochhammer{z}{q}{n} & = 
     \begin{cases} 
     \prod_{j=0}^{n-1} (1-z q^{j}) & \text{ if $n \geq 1$; } \\ 
     1 & \text{ if $n = 0$. } 
     \end{cases} 
\end{align*} 
For $0 \leq m \leq n$, we define the 
\emph{Gaussian polynomial}, or \emph{$q$-binomial coefficient}, 
$\QBinomial{n}{m}{q}$, exactly as follows \citep[\S 17.2; \S 26.9]{NISTHB}: 
\begin{align*} 
\tag{$q$-Binomial Coefficients} 
\QBinomial{n}{m}{q} & := 
     \frac{\QPochhammer{q}{q}{n}}{\QPochhammer{q}{q}{m} 
     \QPochhammer{q}{q}{n-m}} \\ 
     & \phantom{:} = 
     \begin{cases} 
          \frac{(1-q^n)(1-q^{n-1}) \cdots (1-q^{n-m+1})}{ 
          (1-q)(1-q^2) \cdots (1-q^m)} & \text{ if $2 \leq m \leq n$; } \\ 
          1 + q + q^2 + \cdots + q^{n-1} & 
          \text{ if $m = 1$; } \\ 
          1 & \text{ if $m = 0$; } \\ 
          0 & \text{ otherwise. } 
     \end{cases} 
\end{align*} 
The $q$-Pochhammer symbol is implicitly expanded in powers of $z$ by the 
\emph{$q$-binomial coefficients}, $\QBinomial{n}{m}{q}$, as 
\begin{align*} 
\QPochhammer{z}{q}{n} & = (1-z)(1-zq) \cdots (1-z q^{n-1}) \\ 
\tag{$q$-Binomial Theorem} 
     & = 
     \sum_{0 \leq i \leq n} \QBinomial{n}{i}{q} q^{\binom{i}{2}} (-z)^{i}. 
\end{align*} 
The product-based definition of the $q$-Pochhammer symbol then implies that 
we have the following pair of triangular recurrences defining the 
$q$-binomial coefficients where $\QBinomial{n}{m}{q} = \QBinomial{n}{n-m}{q}$ 
for all $0 \leq m \leq n$ \citep[\S 17.2(ii)]{NISTHB}: 
\begin{align} 
\notag 
\QBinomial{n}{m}{q} & = 
     \QBinomial{n-1}{m-1}{q} + q^{m\phantom{-n}} \QBinomial{n-1}{m}{q} \\ 
\label{eqn_QBinomialCoeff_triangular_recs} 
\QBinomial{n}{m}{q} & = 
     \QBinomial{n-1}{m}{q} + q^{n-m} \QBinomial{n-1}{m-1}{q}. 
\end{align} 
\end{definition} 

\begin{prop}[Denominator Convergent Function Formula]
\label{prop_DenomConvFn_Seq_Formula} 
For fixed non-zero $z, q \in \mathbb{C}$ and integers $h \geq 0$, 
we have an exact representation of the denominator convergent functions, 
$\ConvQ_h(q, z)$, expanded by the $q$-binomial coefficients, 
$\QBinomial{n}{m}{q^2}$, from Definition \ref{def_qBinom_qPH} 
in the following form: 
\begin{align*} 
Q_h(q, z) & = \sum_{0 \leq i \leq h} 
     \QBinomial{h}{i}{q^2} q^{(2h-1) i} (-z)^i. 
\end{align*} 
\end{prop} 
\begin{proof} 
The proof follows from \eqref{eqn_ConvFn_PhzQhz_rdefs} and 
repeated applications of \eqref{eqn_QBinomialCoeff_triangular_recs} by 
induction on $h$. In particular, we first notice that 
Table \ref{table_ConvFnQhz_spcases} implies that the right-hand-side 
formula stated by the proposition holds when $h := 0, 1$. 
For some $h \geq 2$, we then assume that the right-hand-side formula for 
$\ConvQ_m(q, z)$ in the proposition statement 
is correct for all $0 \leq m < h$, which, in particular, implies that the 
stated formula holds for $\ConvQ_{h-1}(q, z)$ and $\ConvQ_{h-2}(q, z)$ in the 
recurrence relation given in \eqref{eqn_ConvFn_PhzQhz_rdefs}. 
If we shift the indices of summation, this implies that 
\begin{align*} 
\ConvQ_h(q, z) & = 
     \underset{ := \IC_{h,01}(q, z)}{\underbrace{ 
     \QBinomial{h-1}{0}{q^2} + \QBinomial{h-1}{1}{q^2} q^{2h-3} (-z) 
     -[z^1] \left(\sum_{1 \leq i \leq h} 
     \QBinomial{h-1}{i-1}{q^2} q^{(2h-3)(i-1)} 
     (-1)^{i-1} z^i\right) \cdot c_h^{[\sq]}(q) z}} \\ 
     & \phantom{=\quad} +  
     \sum_{2 \leq i \leq h} \Biggl[ 
     \QBinomial{h-1}{i}{q^2} q^{-2i} + 
     \left(q^{4h-3} + q^{4h-5} - q^{2h-3}\right) q^{-2i} 
     \QBinomial{h-1}{i-1}{q^2} \\ 
\tag{i} 
     & \phantom{=\quad+\sum\Biggl[\ } + 
     \left(1-q^{2h-2}\right) q^{2h-4i} \QBinomial{h-2}{i-2}{q^2} 
     \Biggr] q^{(2h-1) i} (-z)^i \\ 
\end{align*} 
We can show that $\IC_{h,01}(q, z) = 1+\QBinomial{h}{1}{q^2} q^{2h-1} (-z)$ by 
expanding the first three terms in the last equation as 
\begin{align*} 
\IC_{h,01}(q, z) & = 
     \QBinomial{h-1}{0}{q^2} + \QBinomial{h-1}{1}{q^2} q^{2h-3} (-z)
     - \QBinomial{h-1}{0}{q^2} q^{2h-3} z \left(q^{2h} + q^{2h-2} - 1\right) \\ 
     & = 
     1 - q^{2h-1} z \cdot \left[\QBinomial{h-1}{1}{q^2} + 
     q^{2h} + q^{2h-2} - 1\right] \cdot q^{-2} \\ 
     & = 
     1 - q^{2h-1} z \times q^{-2} \cdot \left(1 + q^2 + q^4 + \cdots + 
     q^{2(h-2)} + q^{2h-2} + q^{2h} - 1\right) \\ 
     & = 
     1 + \QBinomial{h-1}{1}{q^2} q^{2h-1} (-z). 
\end{align*} 
We complete the proof by simplifying the coefficients of the 
powers of $(-z)^i$ for $2 \leq i \leq h$ on the right-hand-side of (i). 
More precisely, we show that for $i \geq 2$, we have the next equation in 
(ii), which we see completes the proof of our stated result. 
\begin{align*} 
\tag{ii} 
\QBinomial{h}{i}{q^2} & = 
     \QBinomial{h-1}{i}{q^2} q^{-2i} + 
     \left(q^{2h} + q^{2h-2} - 1\right) q^{-2i} 
     \QBinomial{h-1}{i-1}{q^2} + 
     \left(1-q^{2h-2}\right) q^{2h-4i} \QBinomial{h-2}{i-2}{q^2} 
\end{align*} 
We will require the next few lemmas following immediately as one and two-line 
consequences of the two recurrence relations for the $q$-binomial coefficients 
given in \eqref{eqn_QBinomialCoeff_triangular_recs}. 
\begin{align*} 
\tag{iii} 
q^{-2i} \left(q^{2h-2} - 1\right) \QBinomial{h-2}{i-1}{q^2} & = 
     q^{-2i} (q^{2h-2}-1)\left[\QBinomial{h-1}{i-1}{q^2} - 
     q^{2(h-i)} \QBinomial{h-2}{i-2}{q^2}\right] \\ 
\tag{iv} 
q^{2(h-i)} \QBinomial{h-1}{i-1}{q^2} & = 
     \QBinomial{h}{i}{q^2} - \QBinomial{h-1}{i}{q^2} \\ 
\tag{v} 
q^{2(h-1-i)} \QBinomial{h-2}{i-1}{q^2} & = 
     \QBinomial{h-1}{i}{q^2} - \QBinomial{h-2}{i}{q^2} \\ 
\tag{vi} 
q^{-2i} \QBinomial{h-1}{i}{q^2} & = 
     q^{-2i} \QBinomial{h-2}{i-1}{q^2} + \QBinomial{h-2}{i}{q^2} \\ 
\tag{vii} 
q^{2(h-i)} \QBinomial{h-1}{i-1}{q^2} & = 
     \QBinomial{h}{i}{q^2} - \QBinomial{h-1}{i}{q^2} 
\end{align*} 
We then proceed by simplifying the right-hand-side of (ii), 
denoted in shorthand notation by $\RHSOp[(ii)]$, as follows: 
\begin{align*} 
\RHSOp[(ii)] & = 
     q^{-2i} \QBinomial{h-1}{i}{q^2} + q^{2(h-i)} \QBinomial{h-1}{i-1}{q^2} + 
     q^{-2i} \left(q^{2h-2}-1\right) \QBinomial{h-2}{i-1}{q^2} && 
     \text{ by (iii) } \\ 
     & = 
     \QBinomial{h-2}{i}{q^2} + q^{2(h-i)} \QBinomial{h-1}{i-1}{q^2} + 
     q^{2(h-1-i)} \QBinomial{h-2}{i-1}{q^2} && 
     \text{ by (vi)} \\ 
     & = 
     \QBinomial{h}{i}{q^2} && \text{ by (iv), (v), and (vii). } 
     \qedhere 
\end{align*} 
\end{proof} 

\begin{remark}[A Few Immediate Consequences] 
\label{remark_Convhqz_recs_and_finite_sums}
The formula for the exact expansion of the convergent functions, 
$\ConvQ_h(q, z)$, for each $h \geq 1$ proved in 
Proposition \ref{prop_DenomConvFn_Seq_Formula} 
provides several immediate identities for the $h^{th}$ convergent functions, 
$\Conv_h(q, z)$, expressed through the denominator convergent functions 
\citep[\cf \S 1.12(ii)]{NISTHB}. 
In particular, for the \emph{$h^{th}$ modulus} sequence, 
$\lambda_h(q) := a_2 a_3 \cdots a_{h+1}$ with 
$a_h := \ab_h^{[\sq]}(q) \cdot z^2$, in \citep[\S 3]{FLAJOLET80B} 
we have the following recurrence relation and exact expansion of the 
$h^{th}$ convergent function, $\Conv_h(q, z)$: 
\begin{align} 
\notag 
\Conv_h(q, z) & = \Conv_{h-1}(q, z) + 
     \frac{(-1)^{h-1} q^{(3h-4)(h-1)} \QPochhammer{q^2}{q^2}{h-1} z^{2h-2}}{ 
     \ConvQ_{h-1}(q, z) \ConvQ_h(q, z)} \\ 
\label{eqn_Convhqz_finite_q-series_sum_repr}
     & = 
     \sum_{1 \leq i \leq h} 
     \frac{(-1)^{i-1} q^{(3i-4)(i-1)} \QPochhammer{q^2}{q^2}{i-1} z^{2i-2}}{ 
     \ConvQ_{i-1}(q, z) \ConvQ_i(q, z)}. 
\end{align} 
We are interested in the forms of the $q$-series expansions of the 
convergent functions, $\Conv_h(q, q^2)$, in the applications discussed in 
Section \ref{Section_Congruences_GFs} of the article below. 
Notice that Proposition \ref{prop_DenomConvFn_Seq_Formula} 
applied to the forms of the previous equation implies that we have the 
exact formulas given by 
\begin{align*} 
\Conv_h(q, q^2) & = 
     \sum_{1 \leq i \leq h} 
     \frac{(-1)^{i-1} q^{3i(i-1)} \QPochhammer{q^2}{q^2}{i-1}}{ 
     \ConvQ_{i-1}(q, q^2) \ConvQ_i(q, q^2)} \\ 
     & = 
     \sum_{1 \leq i \leq h} 
     \frac{(-1)^{i-1} q^{3i(i-1)} \QPochhammer{q^2}{q^2}{i-1}}{ 
     \sum\limits_{0 \leq n < 2i} \left(\sum\limits_{0 \leq j \leq n} 
     \QBinomial{i}{j}{q^2} 
     \QBinomial{i-1}{n-j}{q^2} q^{2j} \right) (-q^{2i-1})^n}. 
\end{align*} 
\end{remark} 

\begin{prop}[Numerator Convergent Function Formula]
\label{prop_NumConvFn_Seq_Formula} 
For fixed non-zero $z, q \in \mathbb{C}$ and integers $h \geq 0$, 
we have an exact representation of the numerator convergent functions, 
$\ConvP_h(q, z)$, expanded in the following form: 
\begin{align*} 
P_h(q, z) & = \sum_{0 \leq n < h} \left( 
     \sum_{0 \leq i \leq n} 
     \QBinomial{h}{i}{q^2} q^{(2h-1) i} (-1)^i q^{(n-i)^2} 
     \right) z^n. 
\end{align*} 
\end{prop} 
\begin{proof} 
If we let $C_{h,h}(q) := [z^n] \ConvP_h(q, z)$, we see immediately from 
\eqref{eqn_ConvFn_PhzQhz_rdefs} that we have a recurrence relation for these 
coefficients given by 
\begin{align} 
\label{eqn_Chnq_rec} 
C_{h,n}(q) & = C_{h-1,n}(q) - q^{2h-3} (q^{2h}+q^{2h-2}-1) C_{h-1,n-1}(q) - 
     q^{6h-10} (q^{2h-2}-1) C_{h-2,n-2}(q). 
\end{align} 
We must prove that for all $h \geq 2$ and $0 \leq k < h$, we have the 
following result: 
\begin{align} 
\label{eqn_Phqz_formula_IHhn} 
C_{h,k}(q) & = \sum_{0 \leq i \leq k} \QBinomial{h}{i}{q^2} 
     q^{(2h-1)i} (-1)^i q^{(k-i)^2}. 
\end{align} 
We proceed by induction on $n$. 
Since $C_{n,k}(q) \equiv 0$ whenever $k < 0$, we see easily from 
\eqref{eqn_Chnq_rec} that we have formulas for all $h \geq 2$ in the 
first two cases of $C_{h,n}$ when $n = 0, 1$ given by 
\begin{align*} 
\tag{i} 
C_{h,0}(q) & = 1 \\ 
\tag{ii} 
C_{h,1}(q) & = C_{h-1,1}(q) - q^{2h-3}(q^{2h}+q^{2h-2}-1) \cdot 1 \\ 
     & = 
     -\sum_{i=2}^h q^{2i-3}(q^{2i}+q^{2i-2}-1) \\ 
     & = 
     q - q^{2h-1}\left(1+q^2+q^4+\cdots+q^{2h-2}\right), 
\end{align*} 
each of which coincide with the formula in \eqref{eqn_Phqz_formula_IHhn} 
whenever $h \geq 2$. 
Let $n \geq 2$ and suppose that \eqref{eqn_Phqz_formula_IHhn} is true for all 
$0 \leq k \leq n-1$ when $h \geq 2$. 
We first show that 
\begin{align*} 
C_{h,n}(q) - C_{h-1,n}(q) & = \sum_{i=1}^{n} \QBinomial{h-1}{i-1}{q^2} \left( 
     q^{2h} + q^{2h-2i} - 1\right) q^{(2h-3)i} (-1)^i q^{(n-i)^2}. 
\end{align*} 
To prove the formula given by the last result, 
we expand the left-hand-side of the previous equation according to 
\eqref{eqn_Chnq_rec} as follows: 
\begin{align*} 
C_{h,n}(q) - C_{h-1,n}(q) & = - q^{2h-3} (q^{2h}+q^{2h-2}-1) C_{h-1,n-1}(q) - 
     q^{6h-10} (q^{2h-2}-1) C_{h-2,n-2}(q) \\ 
     & = 
     - q^{2h-3} (q^{2h}+q^{2h-2}-1) \sum_{i=0}^{n-1} \QBinomial{h-1}{i}{q^2} 
     q^{(2h-3)i} (-1)^i q^{(n-1-i)^2} \\ 
     & \phantom{=} - 
     q^{6h-10} (q^{2h-2}-1) \sum_{i=0}^{n-1} \QBinomial{h-2}{i}{q^2} 
     q^{(2h-5)i} (-1)^i q^{(n-2-i)^2} \\ 
     & = 
     -(q^{2h}+q^{2(h-1)}-1) \QBinomial{h-1}{0}{q^2} q^{2h-3} q^{(n-1)^2} \\ 
     & \phantom{=} + 
     \sum_{i=2}^{n} \left((q^{2h}+q^{2h-2}-1) \QBinomial{h-1}{i-1}{q^2} - 
     q^{2h-2i}(q^{2h-2}-1) \QBinomial{h-2}{i-2}{q^2}\right) \times \\ 
     & \phantom{=+\sum\ } \times 
     q^{(2h-3)i} (-1)^i q^{(n-i)^2}. 
\end{align*} 
It suffices to show that for any 
$h \geq 2$ and all $2 \leq i \leq n < h$, we have that 
\begin{align} 
\label{eqn_proof_Phqz_formula_tag_starstar} 
\QBinomial{h-1}{i-1}{q^2} (q^{2h}+q^{2h-2i}-1) & = 
(q^{2h}+q^{2h-2}-1) \QBinomial{h-1}{i-1}{q^2} - q^{2h-2i}(q^{2h-1}-1) 
     \QBinomial{h-2}{i-2}{q^2}. 
\end{align} 
We complete the proof by proving that \eqref{eqn_proof_Phqz_formula_tag_starstar} 
is correct whenever $h \geq 2$ and $2 \leq i \leq n < h$ using the formulas for the 
$q$-binomial coefficients given in terms of the $q$-Pochhammer symbol in 
Definition \ref{def_qBinom_qPH} as follows: 
\begin{align*} 
\RHSOp[\eqref{eqn_proof_Phqz_formula_tag_starstar}] & = 
     (q^{2h}+q^{2h-2}-1) \QBinomial{h-1}{i-1}{q^2} - q^{2h-2i}(q^{2h-1}-1) 
     \QBinomial{h-2}{i-2}{q^2} \\ 
     & = (q^{2h}+q^{2h-2}-1) \frac{\QPochhammer{q^2}{q^2}{h-1}}{ 
     \QPochhammer{q^2}{q^2}{i-1} \QPochhammer{q^2}{q^2}{h-i}} - 
     \frac{q^{2h-2i} (q^{2(h-1)}-1) \QPochhammer{q^2}{q^2}{h-2}}{ 
     \QPochhammer{q^2}{q^2}{i-2} \QPochhammer{q^2}{q^2}{h-i}} \\ 
     & = 
     \QBinomial{h-1}{i-1}{q^2} \left( (q^{2h}+q^{2h-2}-1) + 
     q^{2h-2i}(1-q^{2i-2}) \right) \\ 
     & = 
     \QBinomial{h-1}{i-1}{q^2} \left(q^{2h}+q^{2h-2i}-1\right) 
\end{align*} 
Thus we have that 
\begin{align*} 
C_{h,n}(q) & = C_{h-1,n}(q) + \sum_{i=1}^{n} \QBinomial{h-1}{i-1}{q^2} \left( 
     q^{2h} + q^{2h-2i} - 1\right) q^{(2h-3)i} (-1)^i q^{(n-i)^2} \\ 
     & = 
     \sum_{j=1}^{h} \sum_{i=1}^{n} \QBinomial{j-1}{i-1}{q^2} \left( 
     q^{2j} + q^{2j-2i} - 1\right) q^{(2j-3)i} (-1)^i q^{(n-i)^2}, 
\end{align*} 
and so to complete the proof, it suffices to show that for all integers 
$k \geq 1$ and all $1 \leq i \leq k$, we have the next result which 
we will prove by induction on $k$ ($h$) below. 
\begin{align} 
\label{eqn_proof_Phqz_formula_tag_starstarstar} 
\sum_{j=1}^k \QBinomial{j-1}{i-1}{q^2} \left(q^{2j}+q^{2j-2i}-1\right) q^{(2j-3)i} 
     & = 
     \QBinomial{k}{i}{q^2} q^{(2k-1)i} 
\end{align} 
When $k = i = 1$, we see that 
\begin{align*} 
\QBinomial{0}{0}{q^2} (q^2+1-1) \cdot q^{-1} & = \QBinomial{1}{1}{q^2} \cdot q = q. 
\end{align*} 
Next, we suppose that \eqref{eqn_proof_Phqz_formula_tag_starstarstar} is true 
for all $k < h$ when $h \geq 2$ and rewrite 
\eqref{eqn_proof_Phqz_formula_tag_starstarstar} according to our hypothesis when 
$k = h$ as follows: 
\begin{align*} 
\LHSOp[\eqref{eqn_proof_Phqz_formula_tag_starstarstar}] & = 
     \sum_{j=1}^{h} \QBinomial{j-1}{i-1}{q^2} \left(q^{2j}+q^{2j-2i}-1\right) 
     q^{(2j-3)i} \\ 
     & = 
     q^{(2h-1)i} \cdot q^{-2i} \left( 
     \QBinomial{h-1}{i}{q^2} + \QBinomial{h-1}{i-1}{q^2} q^{2(h-i)} - 
     (1-q^{2h}) \QBinomial{h-1}{i-1}{q^2}\right) \\ 
     & = 
     q^{(2h-1)i} \cdot q^{-2i} \left( 
     \QBinomial{h}{i}{q^2} - (1-q^{2h}) \QBinomial{h-1}{i-1}{q^2}\right) \\ 
     & = 
     q^{(2h-1)i} \cdot q^{-2i} \QBinomial{h}{i}{q^2} \left( 
     1 - \frac{(1-q^{2h})(1-q^{2i})}{(1-q^{2h})}\right) \\ 
     & = 
     \QBinomial{h}{i}{q^2} q^{(2h-1)i}. 
     \qedhere 
\end{align*} 
\end{proof} 

\subsection{Proof of the Main Theorem} 
\label{subSection_ProofOfMainTheorem} 

\begin{proof}[Proof of Theorem \ref{theorem_main_theorem}]  
Since both of $\ConvP_h(q, z)$ and $\ConvQ_h(q, z)$ finite-degree polynomials 
in $z$ (and in $q$), we see that for all $h \geq 1$, the $h^{th}$ 
convergent functions, $\Conv_h(q, z)$, are rational functions of $z$. 
This fact, together with the observation that 
$\deg_z \left\{ \ConvQ_h(q, z) \right\} = h$ for all $h \geq 0$, 
implies that we have $h$-order finite difference equations for the 
coefficients of $\Conv_h(q, z) \equiv \ConvP_h(q, z) / \ConvQ_h(q, z)$ 
of the following form \citep[\cf \S 2]{GFLECT}: 
\begin{align*} 
\tag{i} 
[z^n] \Conv_h(q, z) & = 
     - \sum_{0 < i \leq \min(n, h)} 
     [z^i] \ConvQ_h(q, z) \cdot [z^{h-i}] \Conv_h(q, z) + 
     [z^n] \ConvP_h(q, z). 
\end{align*} 
We first notice that for any functions, $F(z)$ and $G(z)$, such that the 
ratio of these two functions has a power series expansion in $z$ about $0$, 
we have that $[z^0] F(z) / G(z) = F(0) / G(0)$, which by 
Proposition \ref{prop_DenomConvFn_Seq_Formula} and 
Proposition \ref{prop_NumConvFn_Seq_Formula} then implies that 
$[z^0] \Conv_h(q, z) = 1$ for all $h \geq 1$. 
Then if we assume by induction that for fixed $h \geq 1$ and all 
$0 \leq i < n$ we have that $[z^i] \Conv_h(q, z) = q^{i^2}$, we can 
use the two propositions to rewrite (i) of the previous equation in the 
following forms when $0 \leq n < h$: 
\begin{align*} 
[z^n] \Conv_h(q, z) & = 
     -\sum_{1 \leq i \leq n} 
     \QBinomial{h}{i}{q^2} q^{(2h-1) i} (-1)^i q^{(n-i)^2} \\ 
     & \phantom{=\ } + 
     \left(\sum_{0 \leq i \leq n} \QBinomial{h}{i}{q^2} q^{(2h-1) i} 
     (-1)^i q^{(n-i)^2}\right) \cdot \Iverson{n < h}, 
     \text{ where $0 \leq n-i < h$ $\forall n \leq h$} \\ 
     & = 
     \left\{\QBinomial{h}{i}{q^2} q^{(2h-1) i} 
     (-1)^i q^{(n-i)^2}\right\}\Biggr\rvert_{i=0} \\ 
     & = 
     q^{n^2}. 
\end{align*} 
To prove that $[z^h] \Conv_h(q, z) = q^{n^2}$ when $n \equiv h$, we 
expand (i) using the known \emph{$q$-binomial theorem} identity for the 
$q$-Pochhammer symbol, $\QPochhammer{z}{q}{h}$, stated in 
Definition \ref{def_qBinom_qPH} as follows \citep[\S 17.2(iii)]{NISTHB}: 
\begin{align*} 
[z^n] \Conv_h(q, z) & = 
     -\sum_{1 \leq i \leq h} 
     \QBinomial{h}{i}{q^2} q^{(2h-1) i} (-1)^i q^{(n-i)^2} \\ 
     & = 
     -q^{n^2} \times \sum_{1 \leq i \leq h} \QBinomial{h}{i}{q^2} 
     q^{(2h-2n-1) i} (-1)^i q^{i^2} \\ 
     & = 
     -q^{n^2} \times \sum_{1 \leq i \leq h} \QBinomial{h}{i}{q^2} 
     (-1)^i q^{2 \binom{i}{2}} \\ 
     & = 
     -q^{n^2} \left[ \QPochhammer{1}{q^2}{h} - \QBinomial{h}{0}{q^2}\right] \\ 
     & = 
     q^{n^2}. 
     \qedhere 
\end{align*} 
\end{proof} 

\begin{remark}[Convergence and Stronger Statements of the Theorem] 
\label{remark_StrongerStmts_OfThe_Theorem} 
We note that we actually have a stronger result, which is that 
$[z^n] \Conv_h(q, z) = q^{n^2}$ for all $0 \leq n < 2h$ when $h \geq 1$, 
though we only needed to prove the weaker statement given in 
Theorem \ref{theorem_main_theorem} above to prove the 
correctness of our new square series J-fraction expansions defined in 
\eqref{eqn_SquareSeries_J-Fraction_Expansions} of the introduction. 
We can strengthen the statement of the theorem by adding that 
when $0 < |q|, |z| < 1$, 
$J_{\infty}^{[\sq]}(q, z) := \lim_{h \rightarrow \infty} \Conv_h(q, z) < \infty$, 
exists, and moreover, the corresponding infinite sum in 
\eqref{eqn_Convhqz_finite_q-series_sum_repr} 
converges uniformly to the ordinary form of the 
geometric square series we have already proved convergent 
integral representations for in \citep{SQSERIESMDS}. 
The proof of this statement is given by applying 
\emph{Pringsheim's theorem} on the convergence of more general 
continued fractions when 
$|a_h| := |z^2 \cdot q^{6h-10} (q^{2h-2}-1)|$ and 
$|b_h| := |1 - q^{2h-3} (q^{2h} + q^{2h-2} -1) \cdot z|$ where $|q|, |z| < 1$, 
so that $|b_h| - 1 \geq |a_h|$ for all $h \geq 2$, 
which then implies that $|\Conv_h(q, z)| < 1$ for all large enough $h$ 
\citep[\S 1.12(v)]{NISTHB} \citep[\S II]{WALL-CFRACS}. 

We can then use the \emph{Weierstrass $M$-test} to prove that we can differentiate the 
limiting case of the infinite $q$-series in 
\eqref{eqn_Convhqz_finite_q-series_sum_repr} termwise with respect to $z$. 
We use these two facts, namely 1) that the limit function 
$J_{\infty}^{[\sq]}(q, z)$ -- 
and so the new infinite square series representation in 
\eqref{eqn_Convhqz_finite_q-series_sum_repr} -- exists as 
$h \rightarrow \infty$; and 2) that this 
series is uniformly convergent in $z$, to prove new forms of infinite 
$q$-series expansions of these ordinary square series when $|z| < 1$, or 
when the parametric $z := \pm c \cdot q^{pm}$ depends on powers of the 
auxiliary square series parameter $|q| < 1$. 
\end{remark} 

\subsection{Consequences of the Main Theorem} 
\label{subSection_Proofs_ConsequencesOfTheMainTheorem}

\begin{prop}[New Infinite Series Representations of the Geometric Square Series] 
\label{prop_InfiniteGeomSqSeries_qSeries_repr} 
For fixed $q, z \in \mathbb{C}$ defined such that $0 < |q|, |z| < 1$, 
we have the following infinite series representation of the 
ordinary geometric square series, $J_{\infty}^{[\sq]}(q, z)$: 
\begin{align*} 
J_{\infty}^{[\sq]}(q, z) & = \sum_{i=1}^{\infty} 
     \frac{(-1)^{i-1} q^{(3i-4)(i-1)} \QPochhammer{q^2}{q^2}{i-1} z^{2i-2}}{ 
     \sum\limits_{0 \leq j \leq n < 2i} 
     \QBinomial{i}{j}{q^2} \QBinomial{i-1}{n-j}{q^2} q^{2j} \cdot 
     \left(-q^{2i-3} z\right)^n}. 
\end{align*} 
\end{prop} 
\begin{proof} 
We employed known expansions of the convergents to arbitrary 
continued fractions cited for reference in \citep[\S 1.12(ii)]{NISTHB} to 
obtain the finite sum representations of $\Conv_h(q, z)$ stated in 
\eqref{eqn_Convhqz_finite_q-series_sum_repr}, which it is 
easy to see hold for each finite $h \geq 1$. 
We know by the discussion above in 
Remark \ref{remark_StrongerStmts_OfThe_Theorem} that the 
infinite limit, $\lim_{h \rightarrow \infty} \Conv_h(q, z)$, exists, 
which implies that we may form a new infinite series expansion for 
$J_{\infty}^{[\sq]}(q, z)$ by letting $h \longrightarrow \infty$ in 
\eqref{eqn_Convhqz_finite_q-series_sum_repr}. 
We then complete the proof by forming a discrete convolution of the 
coefficients of the polynomial expansions for the 
convergent denominator functions, $\ConvQ_h(q, z)$, in $z$ proved as in the 
formula from Proposition \ref{prop_DenomConvFn_Seq_Formula}. 
\end{proof} 

Another class of series representations for 
special functions involves taking the derivatives of $J_{\infty}^{[\sq]}(q, z)$ 
termwise with respect to $z$. The first-order case of such derivatives of 
our special square series expansions with a corresponding 
integral representation proved in \citep{SQSERIESMDS} given by 
\begin{align*} 
\vartheta_{0,1}(a, b; q, z) & := \sum_{n} (an+b) q^{n^2} z^n \\ 
     & \phantom{:} = 
\int_0^{\infty} 
     \frac{2acz e^{-t^2/2}}{\sqrt{2\pi}} \left[ 
     \frac{\left(c^2 z^2 + 1\right) \cosh\left(t \sqrt{2 \Log(q)}\right) - 2cz}{ 
     \left(c^2 z^2 - 2cz \cosh\left(t \sqrt{2 \Log(q)}\right) + 1\right)^2} 
     \right] dt + 
     b \cdot J_{\infty}^{[\sq]}(q, z), 
\end{align*} 
arises frequently in applications, 
and so is explicitly considered as the particular case in the next corollary. 
Other particular explicit higher-order derivative expansions analogous to the 
formula cited in the 
previous equation for $\vartheta_{0,m}(a, b; q, z)$ when $m \geq 2$ is 
integer-valued are of course similarly easy to obtain directly by repeated 
differentiation of the standard quotient rule. 

\begin{cor}[New Series Representations for Derivatives of the Geometric Square Series] 
\label{cor_InfiniteGeomSqSeries_qSeries_repr_anpb_derivs} 
Suppose that $q, z \in \mathbb{C}$ are defined such that 
$0 < |q|, |z| < 1$, and that the component series functions, 
$f_{q,i}(z)$ and $g_{q,i}(z)$, are defined for integers $i \geq 1$ by the 
next formulas. 
\begin{align*} 
f_{q,i}(z) & = (-1)^{i-1} q^{(3i-4)(i-1)} 
     \QPochhammer{q^2}{q^2}{i-1} z^{2i-2} \\ 
f_{q,i}^{\prime}(z) & = (2i-2) \cdot (-1)^{i-1} q^{(3i-4)(i-1)} 
     \QPochhammer{q^2}{q^2}{i-1} z^{2i-3} \\ 
g_{q,i}(z) & = \sum_{0 \leq j \leq n < 2i} 
     \QBinomial{i}{j}{q^2} \QBinomial{i-1}{n-j}{q^2} q^{2j} \cdot 
     \left(-q^{2i-3} z\right)^n \\ 
g_{q,i}^{\prime}(z) & = \sum_{0 \leq j \leq n < 2i} 
     \QBinomial{i}{j}{q^2} \QBinomial{i-1}{n-j}{q^2} n \cdot q^{2j} \cdot 
     \left(-q^{2i-3}\right)^{n} z^{n-1} 
\end{align*} 
Then for fixed scalars $a, b \in \mathbb{C}$ (not both zero), 
we have a new infinite series for the ordinary square series 
generating functions, $\vartheta_{0,1}(a, b; q, z)$, in the form of 
\begin{align*} 
\vartheta_{0,1}(a, b; q, z) & = 
     a z \times \sum_{i=1}^{\infty} \frac{f_{q,i}^{\prime}(z) g_{q,i}(z) - 
     f_{q,i}(z) g_{q,i}^{\prime}(z)}{\left[g_{q,i}(z)\right]^2} + 
     b \cdot J_{\infty}^{[\sq]}(q, z). 
\end{align*} 
\end{cor} 
\begin{proof} 
The result follows directly from the quotient rule applied to 
Proposition \ref{prop_InfiniteGeomSqSeries_qSeries_repr}. 
We note that we may differentiate the functions of $z$ in the series 
from the previous proposition term by term over each $i$ 
since we have already shown in 
Remark \ref{remark_StrongerStmts_OfThe_Theorem} that this 
limiting case of the $h^{th}$ convergent function series in 
\eqref{eqn_Convhqz_finite_q-series_sum_repr} 
converges uniformly as a function of $z$. 
\end{proof} 

\section{Applications} 
\label{Section_Applications} 
\label{Section_Congruences_GFs}

\begin{table}[ht]
\centering 
\begin{tabular}{|c|l|} \hline 
$h$ & $1 + 2q \cdot \Conv_h(q, q^2)$ \\ \hline 
1 & $\frac{(1+q) \left(-1-q+q^2\right)}{(-1+q) \left(1+q+q^2\right)}$ \\
2 & $\frac{(1+q)^2 \left(-1-q-2 q^3-q^4-q^6+q^7\right)}{\left(1+q^2\right) \left(-1-q-q^4+q^6+q^7\right)}$ \\ 
3 & $\frac{(1+q)^3 \left(1-q+q^2\right) \left(1+q+q^3+q^5+q^6-q^7+q^8-q^9-q^{11}-2 q^{12}-q^{14}+q^{15}\right)}{(-1+q) \left(1+q^2\right) \left(1+q+q^2\right) \left(1+q+q^2+q^3+q^4+q^5+q^6\right) \left(-1+q^2-q^3-q^4+q^5-q^7+q^9\right)}$ \\ 
4 & $\frac{(1+q)^4 \left(1-q+q^2\right) \left(1-q^2+2 q^3+q^6+2 q^8-q^9+q^{10}+q^{12}-2 q^{14}+3 q^{15}-2 q^{16}-q^{18}-q^{21}+q^{22}-2 q^{23}+q^{24}\right)}{\left(1+q^4\right) \left(1+q+q^3+q^4+q^6+q^7+q^8+q^9-q^{13}-q^{14}-q^{15}-q^{16}-q^{17}-q^{18}-q^{19}-q^{20}-q^{21}+q^{25}+q^{26}\right)}$ \\ 
\hline 
\end{tabular} 

\bigskip 
\caption{Special Cases of the Convergent Functions, 
         $\widetilde{\vartheta}_{3,h}(q) := 1 + 2q \cdot \Conv_h(q, q^2)$} 
\label{table_ConvFnConvhq2z_spcases} 

\end{table} 

\subsection{New Generating Functions for the Sums of Squares Functions} 
\label{subSection_rpn_NewOGF_reprs}

The result proved in Proposition \ref{prop_InfiniteGeomSqSeries_qSeries_repr} 
implies that in limiting cases we also have new infinite series for the 
ordinary generating functions of $r_p(n)$ for $p \geq 2$ 
in the following forms: 
\begin{align*} 
\vartheta_3(q)^p & = \left(1 + 2q \times \sum_{i=1}^{\infty} 
     \frac{(-1)^{i-1} q^{3i(i-1)} \QPochhammer{q^2}{q^2}{i-1}}{ 
     \sum\limits_{0 \leq n < 2i} \left(\sum\limits_{0 \leq j \leq n} 
     \QBinomial{i}{j}{q^2} 
     \QBinomial{i-1}{n-j}{q^2} q^{2j} \right) (-q^{2i-1})^n} 
     \right)^{p}. 
\end{align*} 
For $0 \leq n \leq (2h-1)^2$ (and occasionally for slightly larger $n$), 
we have that 
\begin{align*} 
r_p(n) = [q^n] \left(1 + 2q \cdot \Conv_h(q, q^2)\right)^p. 
\end{align*} 
Several special cases of the forms of the modified $h^{th}$ 
convergent functions, 
$\widetilde{\vartheta}_{3,h}(q) := 1 + 2q \cdot \Conv_h(q, q^2)$, 
are expanded in factored form in Table \ref{table_ConvFnConvhq2z_spcases}. 
These special case convergent functions also satisfy expansions given by 
\begin{align*} 
\widetilde{\vartheta}_{3,2}(q) & = 
     1+\frac{2 (2+3 q)}{13 \left(1+q^2\right)}-\frac{2 \left(-2+8 q+12 q^2+5 q^3+12 q^4+18 q^5+3 q^6\right)}{13 \left(-1-q-q^4+q^6+q^7\right)} \\ 
\widetilde{\vartheta}_{3,3}(q) & = 
     1-\frac{4}{21 (-1+q)}+\frac{2 (3+4 q)}{25 \left(1+q^2\right)}+\frac{2 (-2+5 q)}{39 \left(1+q+q^2\right)} \\ 
     & \phantom{=1\ } + 
     \frac{2 \left(-31-41 q+82 q^2+44 q^3+90 q^4+66 q^5\right)}{301 \left(1+q+q^2+q^3+q^4+q^5+q^6\right)} \\ 
     & \phantom{=1\ } - 
     \frac{2 \left(-852+8364 q-2721 q^2-5830 q^3+17408 q^4+1596 q^5-10747 q^6+12927 q^7+5761 q^8\right)}{13975 \left(-1+q^2-q^3-q^4+q^5-q^7+q^9\right)} \\ 
\widetilde{\vartheta}_{3,4}(q) & = 
     1-\frac{2 \left(16-64 q-q^2+4 q^3\right)}{257 \left(1+q^4\right)} \\ 
     & \phantom{=1\ } + 
     \scriptstyle{\frac{\vartheta_{3,4}(q)}{\left(257 \left(1+q+q^3+q^4+q^6+q^7+q^8+q^9-q^{13}-q^{14}-q^{15}-q^{16}-q^{17}-q^{18}-q^{19}-q^{20}-q^{21}+q^{25}+q^{26}\right)\right)}}, 
\end{align*} 
where the numerator polynomial, $\vartheta_{3,4}(q)$, is defined by 
\begin{align*} 
\vartheta_{3,4}(q) & := \bigl(2 \bigl(16+209 q+192 q^2+19 q^3+454 q^4+497 q^5+84 q^6+451 q^7+525 q^8+486 q^9+626 q^{10} \\ 
     & \phantom{=2\ } + 
     323 q^{11}+507 q^{12}+526 q^{13}-64 q^{14}-17 q^{15}+309 q^{16}+33 q^{17}-405 q^{18}-195 q^{19} \\ 
     & \phantom{=2\ } - 
     264 q^{20}-502 q^{21}-305 q^{22}-322 q^{23}-254 q^{24}+4 q^{25}\bigr)\bigr). 
\end{align*} 

\begin{remark}[Convergent Function Congruences Versus Generating Functions as Series Over $q$] 
Note that while the J-fraction properties listed in the references 
\citep{FLAJOLET80B,FLAJOLET82,GFLECT,MULTIFACT-CFRACS} 
provide congruences for the coefficients of $\Conv_h(q, z)$ over $z$ 
modulo $h$ for integers $h \geq 2$, we are unable to find exact 
generating functions for the sums of squares functions, 
$r_p(n) \pmod{h}$, 
which are generated by corresponding powers series expansions over $q$ 
since $[z^n] \Conv_h(q, z) \neq q^{n^2}$ for $n \geq 2h$. 
We instead adapt our expansions from the previous few equations to 
form partial generating functions in the next corollaries by 
reducing the polynomial numerator and denominator term coefficients modulo $h$. 
This procedure follows by the approach to generating the functions, $r_p(n)$, 
recursively for $0 \leq n < h$ by the finite difference equations 
implied by the rationality of $\widetilde{\vartheta}_{3,h}(q)$ in $q$ 
for all $h \geq 2$. 
\end{remark} 

\begin{cor}[Generating Functions Enumerating $r_p(n)$ Modulo $3$] 
For $0 \leq n \leq 36$ and integers $p \geq 2$, we have that 
\begin{align*} 
r_p(n) & \equiv [q^n] \left(1+\frac{1+2 q+2 q^2+q^3}{1+q+q^2+q^3+q^4+q^5+q^6}+\frac{2+q+q^3+2 q^7+2 q^8+2 q^{11}+q^{12}}{1+2 q^9+2 q^{11}+q^{14}}\right)^{p} 
     \pmod{3}. 
\end{align*} 
\end{cor} 

\begin{cor}[Generating Functions Enumerating $r_p(n)$ Modulo $4$] 
For $0 \leq n \leq 64$ and integers $p \geq 2$, we have that 
\begin{align*} 
r_p(n) & \equiv [q^n] \scriptstyle{\left(1+\frac{2 q^2}{1+q^4}+\frac{2 q+2 q^3+2 q^5+2 q^7+2 q^8+2 q^{11}+2 q^{12}+2 q^{15}+2 q^{16}+2 q^{17}+2 q^{18}+2 q^{19}+2 q^{22}}{1+q+q^3+q^4+q^6+q^7+q^8+q^9+3 q^{13}+3 q^{14}+3 q^{15}+3 q^{16}+3 q^{17}+3 q^{18}+3 q^{19}+3 q^{20}+3 q^{21}+q^{25}+q^{26}}\right)^{p}}  
     \pmod{4}. 
\end{align*} 
\end{cor} 

\begin{remark}[Generating Functions Enumerating the Sum of Divsors Function] 
Since $r_4(2k+1) = 8 \cdot \sigma_1(2k+1)$ for all integers $k \geq 0$, 
we may employ the first infinite $q$-series expansion given in this 
section to enumerate the \emph{sum of divisors function}, 
$\sigma_1(n) = \sum_{d | n} d$, 
over odd positive integers $n$ as follows: 
\begin{align*} 
\sum_{k=0}^{\infty} \sigma_1(2k+1) q^{2k+1} & = 
     \frac{1}{16} \cdot \left(\vartheta_3(q)^4 - \vartheta_3(-q)^4\right) \\ 
     & = 
     \sum_{b = \pm 1} \frac{b}{16}\left( 
     1 + 2bq \times \sum_{i=1}^{\infty} \frac{(-1)^{i-1} (bq)^{3i(i-1)} 
     \QPochhammer{q^2}{q^2}{i-1}}{\sum\limits_{0 \leq j \leq n < 2i} 
     \QBinomial{i}{j}{q^2} \QBinomial{i-1}{n-j}{q^2} q^{2j} 
     (-b q^{2i-1})^n} \right)^{4}. 
\end{align*} 
We can also use the square series generating function results for 
$J_{\infty}^{[\sq]}(q, z) \equiv \Conv_{\infty}(q, z)$ from 
Remark \ref{remark_Convhqz_recs_and_finite_sums} and 
Proposition \ref{prop_InfiniteGeomSqSeries_qSeries_repr} 
to prove new $q$-series expansions for the infinite $q$-Pochhammer product, 
$\QPochhammer{q}{q}{\infty}$, which generates the sum of divisors function as 
\citep[\S 26.10(iii)]{NISTHB} 
\begin{align*} 
\sum_{n \geq 1} \sigma_1(n) q^n & = -q \cdot d / dq \left[ 
      \Log\QPochhammer{q}{q}{\infty}\right]. 
\end{align*} 
Then we expand $\QPochhammer{q}{q}{\infty}$ in the generating function 
for $\sigma_1(n)$ in the previous equation by the 
unilateral series given by 
\begin{align*} 
\QPochhammer{q}{q}{\infty} & = 1 - \sum_{n=0}^{\infty} (-1)^n \left( 
     q^{(n+1)(3n+2)/2} + q^{(n+1)(3n+4)/2} \right) \\ 
     & = 
     1 - q \times \sum_{n=0}^{\infty} (-q^{5/2})^n q^{3 n^2 / 2} - 
     q^2 \times \sum_{n=0}^{\infty} (-q^{7/2})^n q^{3 n^2 / 2} \\ 
     & = 
     1 - q \times \sum_{i=1}^{\infty} 
     \frac{(-1)^{i-1} q^{(9i-2)(i-1)/2} \QPochhammer{q^3}{q^3}{i-1}}{ 
     \sum\limits_{0 \leq j \leq n < 2i} \QBinomial{i}{j}{q^3} 
     \QBinomial{i-1}{n-j}{q^3} q^{3j} \cdot q^{(3i-2) n}} \\ 
     & \phantom{=1\ } - 
     q^2 \times \sum_{i=1}^{\infty} 
     \frac{(-1)^{i-1} q^{(9i+2)(i-1)/2} \QPochhammer{q^3}{q^3}{i-1}}{ 
     \sum\limits_{0 \leq j \leq n < 2i} \QBinomial{i}{j}{q^3} 
     \QBinomial{i-1}{n-j}{q^3} q^{3j} \cdot q^{(3i-1) n}}. 
\end{align*} 
Similarly, the \emph{partition function}, $p(n)$, is generated by the 
$q$-series expansion of the reciprocal of the previous series for 
$\QPochhammer{q}{q}{\infty}$. 
Other applications of the new infinite $q$-series provided by the 
limiting cases of these sums for the $h^{th}$ convergent functions are 
given in the next section. 
\end{remark} 

\subsection{New Infinite Series for Special Theta Functions} 
\label{subSection_JThetaFns_Newq-Series_reprs}

The \emph{Jacobi theta functions}, denoted by 
$\vartheta_i(u, q)$ for $i = 1,2,3,4$, or by 
$\vartheta_i\left(u \vert \tau\right)$ when the 
\emph{nome} $q \equiv \exp\left(\imath\pi\tau\right)$ 
satisfies $\Im(\tau) > 0$, 
form another class of square series expansions of 
special interest in the applications considered in \citep{SQSERIESMDS}. 
The classical forms of these theta functions satisfy the 
respective \emph{bilateral} and corresponding \emph{asymmetric}, \emph{unilateral} 
\emph{Fourier--type} square series expansions given by 
\citep[\S 20.2(i)]{NISTHB} 
\begin{align} 
\label{eqn_JThetaFn_series_defs}
\vartheta_1(u, q) & = 
     \sum_{n=-\infty}^{\infty} 
     q^{\left(n+\frac{1}{2}\right)^2} (-1)^{n-1/2} e^{(2n+1) \imath u} && = 
     2 q^{1/4} \sum_{n=0}^{\infty} q^{n(n+1)} (-1)^{n} 
     \sin\left((2n+1) u\right) \\ 
\notag 
\vartheta_2(u, q) & = 
     \sum_{n=-\infty}^{\infty} 
     q^{\left(n+\frac{1}{2}\right)^2} e^{(2n+1) \imath u} && = 
     2 q^{1/4} \sum_{n=0}^{\infty} q^{n(n+1)} 
     \cos\left((2n+1) u\right) \\ 
\notag 
\vartheta_3(u, q) & = 
     \sum_{n=-\infty}^{\infty} 
     q^{n^2} e^{2n \imath u} && = 
     1 + 2 \sum_{n=1}^{\infty} q^{n^2} \cos\left(2n u\right) \\ 
\notag 
\vartheta_4(u, q) & = 
     \sum_{n=-\infty}^{\infty} 
     q^{n^2} (-1)^{n} e^{2n \imath u} && = 
     1 + 2 \sum_{n=1}^{\infty} q^{n^2} (-1)^{n} \cos\left(2n u\right). 
\end{align} 

\begin{cor}[New $q$-Series Expansions of the Jacobi Theta Functions] 
\label{cor_Newq-SeriesExps_JThetaFns} 
For $u \in \mathbb{R}$ and fixed $q \in \mathbb{C}$ such that 
$0 < |q| < 1$, we have new $q$-series representations of the 
Fourier series for the Jacobi theta functions expanded in the 
following forms: 
\begin{align*} 
\vartheta_1(u, q) & = 
     \sum_{b=\pm 1} \frac{q^{1/4} e^{\imath b u}}{b\imath} \times \left[ 
     \sum_{i=1}^{\infty} 
     \frac{(-1)^{i-1} q^{(3i-4)(i-1)} \QPochhammer{q^2}{q^2}{i-1} 
     \left(-e^{2\imath bu} q\right)^{2i-2}}{\sum\limits_{0 \leq j \leq n < 2i} 
     \QBinomial{i}{j}{q^2} \QBinomial{i-1}{n-j}{q^2} q^{2j} \cdot 
     \left(q^{i-1} e^{\imath bu}\right)^{2n}} 
     \right] \\ 
\vartheta_2(u, q) & = 
     \sum_{b=\pm 1} q^{1/4} e^{\imath b u} \times \left[ 
     \sum_{i=1}^{\infty} 
     \frac{(-1)^{i-1} q^{(3i-4)(i-1)} \QPochhammer{q^2}{q^2}{i-1} 
     \left(q e^{2\imath bu}\right)^{2i-2}}{\sum\limits_{0 \leq j \leq n < 2i} 
     \QBinomial{i}{j}{q^2} \QBinomial{i-1}{n-j}{q^2} q^{2j} \cdot 
     \left(q^{i-1} e^{\imath bu}\right)^{2n}} 
     \right] \\ 
\vartheta_3(u, q) & = 
     1 + \sum_{b=\pm 1} q e^{2\imath b u} \times \left[ 
     \sum_{i=1}^{\infty} 
     \frac{(-1)^{i-1} q^{(3i-4)(i-1)} \QPochhammer{q^2}{q^2}{i-1} 
     \left(q e^{\imath bu}\right)^{4i-4}}{\sum\limits_{0 \leq j \leq n < 2i} 
     \QBinomial{i}{j}{q^2} \QBinomial{i-1}{n-j}{q^2} q^{2j} \cdot 
     \left(-q^{2i-1} e^{2\imath bu}\right)^{n}} 
     \right] \\ 
\vartheta_4(u, q) & = 
     1 - \sum_{b=\pm 1} q e^{2\imath b u} \times \left[ 
     \sum_{i=1}^{\infty} 
     \frac{(-1)^{i-1} q^{(3i-4)(i-1)} \QPochhammer{q^2}{q^2}{i-1} 
     \left(-q e^{\imath bu}\right)^{4i-4}}{\sum\limits_{0 \leq j \leq n < 2i} 
     \QBinomial{i}{j}{q^2} \QBinomial{i-1}{n-j}{q^2} q^{2j} \cdot 
     \left(q^{2i-1} e^{2\imath bu}\right)^{n}} 
     \right]. 
\end{align*} 
\end{cor} 
\begin{proof} 
We can expand the sine and cosine functions by exponentials thus 
reducing the Fourier series terms in \eqref{eqn_JThetaFn_series_defs} 
of the form $\cos((an+b)u)$ and $\sin((an+b)u)$ to $n^{th}$ powers of the 
exponential function as 
\begin{align*} 
\cos((an+b)u) & = \frac{1}{2}\left( 
     e^{(an+b) \imath u} + e^{-(an+b) \imath u}\right) \\ 
\sin((an+b)u) & = \frac{1}{2\imath}\left( 
     e^{(an+b) \imath u} - e^{-(an+b) \imath u}\right) 
\end{align*} 
Each of these series then follows as a special case of the new infinite 
series representations proved by 
Proposition \ref{prop_InfiniteGeomSqSeries_qSeries_repr}. 
\end{proof} 

Additional square series expansions related to these forms are derived 
from the special cases of the Jacobi theta functions 
defined by $\vartheta_i(q) \equiv \vartheta_i(0, q)$, and by 
$\vartheta_i^{(j)}(u, q) \equiv 
 \partial^{(j)}{\vartheta_i(u_0, q)} / \partial{u_0}^{(j)} \vert_{u_0=u}$ 
or by $\vartheta_i^{(j)}(q) \equiv \vartheta_i^{(j)}(0, q)$ 
for the higher--order $j^{th}$ derivatives taken over 
any fixed $j \in \mathbb{Z}^{+}$. 
For example, we can employ the result for the first-order case in 
Corollary \ref{cor_InfiniteGeomSqSeries_qSeries_repr_anpb_derivs} to 
expand new $q$-series representations for the special case series given by 
\begin{align*} 
\vartheta_1^{\prime}(q) & = 2 q^{1/4} \times \sum_{n=0}^{\infty} 
     (2n+1) q^{n^2} (-q)^n && = 
     2 q^{1/4} \cdot \vartheta_{0,1}(2, 1; q, -q) \\ 
(q)_{\infty}^{3} & = 1 - q \times \sum_{n=0}^{\infty} 
     (-1)^n (2n+3) q^{n(n+3)/2} && = 
     1 - q \cdot \vartheta_{0,1}\left(2, 3; \sqrt{q}, q^{3/2}\right). 
\end{align*} 
We can also form integral representations over the variant forms of the 
Jacobi theta functions given by 
$\vartheta_i(0 | \imath x^2) \equiv \vartheta_i\left(e^{-\pi x^2}\right)$ 
as considered in the examples from \citep{SQSERIESMDS}. 
For example, when $\Re(s) > 2$, we have that \citep[\S 20.10(i)]{NISTHB} 
\begin{align*} 
\frac{(2^s - 1)}{\sqrt{\pi^s}} & \Gamma\left(\frac{s}{2}\right) \zeta(s) = 
     \int_0^{\infty} x^{s-1} \vartheta_2(0, \imath x^2) dx \\ 
     & = 
     \sum_{i=1}^{\infty} \int_0^{\infty} \left[ 
     \frac{2 x^{s-1} e^{-\pi x^2 / 4} (-1)^{i-1} e^{-\pi x^2 (3i-4)(i-1)} 
     \QPochhammer{e^{-2\pi x^2}}{e^{-2\pi x^2}}{i-1} 
     e^{-2\pi x^2 (i-1)}}{\sum\limits_{0 \leq j \leq n < 2i} 
     \QBinomial{i}{j}{e^{-2\pi x^2}} \QBinomial{i-1}{n-j}{e^{-2\pi x^2}} 
     \times e^{-2\pi x^2 \cdot j} \left(-e^{2\pi x^2 (i-1)}\right)^n} 
     \right]\ dx. 
\end{align*} 

\begin{remark}[The General Two-Variable Ramanujan Theta Function] 
For any non-zero $a, b \in \mathbb{C}$ such that $|ab| < 1$, the 
\emph{general two-variable Ramanujan theta function}, $f(a, b)$, is 
expanded by the unilateral series representations 
\begin{align*} 
f(a, b) & = 1 + \sum_{n=1}^{\infty} \left[ 
     a^{n(n+1)/2} b^{n(n-1)/2} + a^{n(n-1)/2} b^{n(n+1)/2} 
     \right] \\ 
     & = 
     1 + a \cdot J_{\infty}^{[\sq]}\left(\sqrt{ab}, a \sqrt{ab}\right) + 
         b \cdot J_{\infty}^{[\sq]}\left(\sqrt{ab}, b \sqrt{ab}\right) \\ 
     & = 
     1 + \sum_{c \in \{a, b\}} \sum_{i=1}^{\infty} 
     \frac{c \cdot (-1)^{i-1} (ab)^{(3i-2)(i-1)} 
     \QPochhammer{ab}{ab}{i-1} c^{2i-2}}{\sum\limits_{0 \leq j \leq n < 2i} 
     \QBinomial{i}{j}{ab} \QBinomial{i-1}{n-j}{ab} (ab)^j \cdot 
     \left(-(ab)^{i-1} \cdot c\right)^{n}}, 
\end{align*} 
which then implies the corresponding new series representations for its 
special case representations of other special functions. 
In particular, we have notable special cases of the series in the 
last equation given by the infinite $q$-Pochhammer product 
$(q)_{\infty} \equiv f(q) := f(-q, -q^2)$, including an immediate 
application to expanding new series for the 
\emph{Dedekind eta function}, $\eta(\tau)$, and the special cases of 
\emph{Ramanujan's functions}, $\varphi(q) \equiv f(q, q)$ and 
$\psi(q) \equiv f(q, q^3)$. 

Explicit special constant values forming real-valued multiples of 
rational inputs to reciprocals of the gamma function are generated by the 
expansions of the inputs of $q := e^{-k\pi}$ to Ramanujan's functions, 
$\varphi(q) \equiv \vartheta_3(q)$ and 
$\psi(q) \equiv 1 / 2 q^{-1/8} \cdot \vartheta_2\left(\sqrt{q}\right)$, 
for the positive real-valued $k$. 
The special case series and their corresponding square series 
integral representations from \citep{SQSERIESMDS} expanded in 
\eqref{eqn_PairOfExamples_ExplicitSpConstantValues_JThetaFnSpCases} of the 
introduction are only two of a number of special case pairs of these 
special-function-defined constant forms for the particular real-valued 
$k \in \{1,2,3,5\} \bigcup \{1,2,\frac{1}{2}\} \bigcup 
       \{\sqrt{2}, \sqrt{3}, \sqrt{6} \} \bigcup \mathbb{Z}^{+}$ 
cited as examples in the original square series transformations article. 
\end{remark} 

\section{Conclusions} 

\subsection{Summary} 

We arrived at formulas for the ``\emph{square series}'' 
subsequences, $c_h^{[\sq]}(q)$ and $\ab_h^{[\sq]}(q)$, 
implicit to the definition of the infinite J-fraction expansion in 
\eqref{eqn_J-Fraction_Expansions} and in 
\eqref{eqn_SquareSeries_J-Fraction_Expansions} by computation rather than the 
combinatorial intuition that motivates many of the 
J-fraction sequence examples cited in the references 
\citep{FLAJOLET80B,FLAJOLET82}. 
In Section \ref{Section_Proofs}, the 
J-fraction series expanded by the ratio of the $h^{th}$ 
convergent functions, $\ConvP_h(q, z)$ and $\ConvQ_h(q, z)$, is proved to 
generate the ordinary geometric square series up to order $h+1$ (or order $2h$) for 
each finite $h \geq 1$ with the limiting case as $h \rightarrow \infty$ 
converging exactly to the ordinary square series generating function, 
$J_{\infty}^{[\sq]}(q, z)$. 
Moreover, we are able to generate higher-order square series expansions by 
differentiating the limiting convergent function termwise with respect to $z$. 

We use known properties of the convergent expansions of more general 
continued fraction representations to prove a new finite sum 
representation of $\Conv_h(q, z)$ given in 
\eqref{eqn_Convhqz_finite_q-series_sum_repr} 
for each $h \geq 1$. 
Since the limit of $\Conv_h(q, z)$ over $h$ exists as $h$ tends to infinity, 
these new finite sum results imply new infinite $q$-series representations 
of both of the ordinary square series cases of $J_{\infty}^{[\sq]}(q, z)$ and 
$J_{\infty}^{[\sq]}\left(q^m, c \cdot q^{p}\right)$ proved in 
Section \ref{subSection_Proofs_ConsequencesOfTheMainTheorem}. The latter ordinary 
square series variant leads to new generating functions for the 
special function sequences, $r_p(n)$ and $\sigma_1(n)$, as well as new 
strictly $q$-series representations of the series for the 
Jacobi theta functions, particular notable special cases of the 
Jacobi theta functions, and for the 
general two-variable Ramanujan theta function given in 
Section \ref{Section_Applications}. 

\subsection{Remarks on Computationally Conjecturing New J-Fraction Expansions} 
\label{subSection_Conclusions_Remarks} 

\begin{figure}[ht] 
\centering
\begin{minipage}[t]{0.91\textwidth} 
\begin{Verbatim}[frame=single,numbers=left]  
Clear[c, ab, Phz, Qhz, Conv]
Phz[h_, z_] := Phz[h, z] = 
     If[(h <= 1), KroneckerDelta[h == 1, True], 
                  (1 - c[h]*z)*Phz[h - 1, z] - ab[h]*(z^2)*Phz[h - 2, z]]
Qhz[h_, z_] := Qhz[h, z] = 
     If[(h <= 1), KroneckerDelta[h == 0, True] + 
                  KroneckerDelta[h == 1, True] * (1 - c[1] * z), 
                  (1 - c[h]*z)*Qhz[h - 1, z] - ab[h]*(z^2)*Qhz[h - 2, z]]
Conv[h_, z_] := FS[Phz[h, z]/Qhz[h, z]] 

getSubsequenceValues[upper_, fnq_] := Module[{eqns, vars, cfsols}, 
     eqns = Table[SeriesCoefficient[Conv[upper, z], {z, 0, n}] == 
                  FunctionExpand[fnq[n]], {n, 1, upper}]; 
     vars = Flatten[Table[{c[i], ab[i + 2]}, {i, 0, upper}]];
     cfsols = Solve[eqns, vars][[1]] // Expand // FullSimplify;
     Return[Map[#1[cfsols]&, {FullSimplify, Factor, Apart}]]; 
]; 

fnqFunctions = { 
     Power[q, #^2]&, 
     (Power[q, #^2] / Factorial[#])&
     Power[q, #^3]&, 
     QPochhammer[q, q, #]&, 
     QPochhammer[a, q, #]&, 
     (1 / QPochhammer[q, q, #])&, 
     QPochhammer[z * q^(-n), q, #]&, 
     (1 / QPochhammer[z * q^(-n), q, #])&, 
     Power[q, (#-1)#/2]/QPochhammer[q, q, #]&, 
     QPochhammer[a, q, #] / QPochhammer[q, q, #]&
}; 

Table[{Function->fnq[n], getSubSequenceValues[7, fnq]}, 
      {fnq, fnqFunctions}] // TableForm
\end{Verbatim} 
\end{minipage} 

\caption{Mathematica Code for Conjecturing New $q$-Series-Related 
         J-Fraction Expansions} 
\label{figure_MmCode_ConjNewJFractionExpSeqs} 

\end{figure} 

The listing given in Figure \ref{figure_MmCode_ConjNewJFractionExpSeqs} 
provides a shortened version of the 
working \emph{Mathematica} code we used to first computationally 
conjecture the implicit subsequences involved in our new square series 
J-fraction expansions in \eqref{eqn_SquareSeries_J-Fraction_Expansions}. 
Despite some trial and error with the method illustrated in the figure for 
conjecturing the new J-fraction expansion results, we do not seem to have 
correspondingly simple continued fraction representations for the 
exponential square series case where $f_n(q) := q^{n^2} / n!$, which is 
studied in more detail in the original 
square series transformations article \citep{SQSERIESMDS}. 
The computational methods we used to compute, and subsequently ``guess``, the 
first few explicit 
values of the J-fraction component sequences are easily extended, 
as demonstrated by the example functions in the last table generated in the 
source code from the figure. 

\begin{table}[ht] 
\begin{tabular}{||c||c|l||c|l||} \hline 
$f_n(q)$ & $c_1(q)$ & $c_h(q)$ ($h \geq 2$) & 
           $\ab_h(q)$ ($h \geq 2$) \\ \hline\hline 
$\QPochhammer{a}{q}{n}$ & $1-a$ & $q^{h-1} - a q^{h-2} \left(q^{h} + q^{h-1} - 1\right)$ & 
                          $a q^{2h-4} (a q^{h-2}-1)(q^{h-1}-1)$ \\ 
$\frac{1}{\QPochhammer{q}{q}{n}}$ & $\frac{1}{1-q}$ & 
     $\frac{q^{h-1} \left(q^{h-1} \QBinomial{h-1}{1}{q} - 
      \QBinomial{h-2}{1}{q}\right)}{ 
      \QBinomial{2h-3}{1}{q} (q^{2h-1}-1)}$ & 
     $-\frac{q^{3h-5}}{(q^{2h-3}-1)^2 (1 + q^{h-2} + q^{h-1} + q^{2h-3})}$ \\ 
$\QPochhammer{z q^{-n}}{q}{n}$ & $\frac{q-z}{q}$ & 
     $\frac{q^h - z - qz + q^h z}{q^{2h-1}}$ & 
     $\frac{(q^{h-1}-1) (q^{h-1}-z) \cdot z}{q^{4h-5}}$ \\ 
$\frac{1}{\QPochhammer{z q^{-n}}{q}{n}}$ & $\frac{q}{q-z}$ & 
     $\frac{q^{h-1} \left(q^{2h-2} + z + q^{h-1} z - q^{h} z\right)}{ 
      (q^{2h-3}-z) (q^{2h-1}-z)}$ & 
     $\QBinomial{h-1}{1}{q} \cdot \frac{q^{3h-4}(1-q)(q^{h-2}-z) \cdot z}{ 
      (q^{2h-4}-z) (q^{2h-3}-z)^2 (q^{2h-2}-z)}$ \\ 
\hline\hline 
\end{tabular} 

\bigskip 
\caption{Conjectures for Sequence Formulas in $q$-Series-Related 
         J-Fraction Expansions} 
\label{table_Conclusions_SpecialJ-Fraction_ConjSubSeq_Results} 

\end{table} 

More precisely, if we let $\{ f_n(q) \}_{n=0}^{\infty}$ 
denote an arbitrary sequence depending on the fixed parameter $q$, we 
seek to explore the analogous $q$-series-related J-fraction expansions 
enumerating this sequence in the form of 
\begin{align*} 
J_{\infty}(f; q, z) & := \sum_{n=0}^{\infty} f_n(q) \cdot z^n. 
\end{align*} 
We can always solve for the sequences of implicit coefficients, 
$\{c_i\}$ and $\{\ab_i\}$, in \eqref{eqn_J-Fraction_Expansions} 
which (at least formally) generate the sequence of $f_n(q)$ over $z$. 
Typically we find that for more well-known $q$-series generating function forms 
it is easy to computationally ``\emph{guess}'' exact formulas for these 
sequences, which are each implicit functions of the $f_i(q)$'s, 
in many special case forms. 

\subsubsection*{Conjectures on the J-Fraction Expansions of Related Series} 

We conjecture, but do not offer conclusive proofs, that 
we have related $q$-series generating functions enumerating variants of the 
sequences, $\QPochhammer{a}{q}{n}$, $\QPochhammer{q}{q}{n}$, and 
$1 / \QPochhammer{q}{q}{n}$ over $z$ defined by the 
implicit subsequence listings given in 
Table \ref{table_Conclusions_SpecialJ-Fraction_ConjSubSeq_Results}. 
Examples of special series that result by generating the sequences, 
$f_n(q)$, in the table by the properties resulting from our conjectured 
J-fraction expansions include the following expansions 
\citep[\S 17]{NISTHB} \citep[\S 1.3.5]{SQSERIESMDS}: 
\begin{align*} 
\sum_{n=0}^{\infty} \QPochhammer{z}{q}{n+1} z^n & = 
     1 + \sum_{n=1}^{\infty} (-1)^n \left[ 
     q^{n(3n-1)/2} z^{3n-1} + q^{n(3n+1)/2} z^{3n} 
     \right] \\ 
\tag{$q$-Exponential Function} 
e_q(z) & = \sum_{n=0}^{\infty} \frac{(1-q)^n z^n}{\QPochhammer{q}{q}{n}} \\ 
\tag{$q$-Hypergeometric Function} 
_1\phi_0(0; -; q, z) & = \sum_{n=0}^{\infty} 
     \frac{z^n}{\QPochhammer{q}{q}{n}} = 
     \frac{1}{\QPochhammer{z}{q}{\infty}} \\ 
\tag{Partition Function Generating Function} 
_1\phi_0(0; -; q, q) & = \sum_{n=0}^{\infty} 
     \frac{q^n}{\QPochhammer{q}{q}{n}} = 
     \frac{1}{\QPochhammer{q}{q}{\infty}} \\ 
\tag{Another $q$-Binomial Theorem} 
_1\phi_0(0; -; q, q) & = \QPochhammer{z q^{-n}}{q}{n}. 
\end{align*} 
For the first sequence case in the table, we can extend the methods of proof 
given in this article for the square series generating functions to 
find that 
\begin{align*} 
\sum_{n=0}^{\infty} \QPochhammer{a}{q}{n} z^n & = 
     \sum_{i=1}^{\infty} \frac{a^{i-1} q^{(i-1)(i-2)} 
     \QPochhammer{a}{q}{i-1} \QPochhammer{q}{q}{i-1} z^{2i-2}}{ 
     \sum\limits_{0 \leq j \leq n < 2i} 
     \QBinomial{i}{j}{q} \QBinomial{i-1}{n-j}{q} 
     \QPochhammer{a q^{i-j}}{q}{j} \QPochhammer{a q^{i-1-n+j}}{q}{n-j} 
     q^{\binom{j}{2}+\binom{n-j}{2}} (-z)^n}. 
\end{align*} 
Other examples of new infinite series and 
q-series expansions for the generating functions of the sequences listed 
in Table \ref{table_Conclusions_SpecialJ-Fraction_ConjSubSeq_Results}, 
along with many other special cases not considered within the 
context of the applications cited in this article, are proved similarly. 

\renewcommand{\refname}{References}

\end{document}